\documentclass[10pt]{amsart}

\usepackage{sasha_ams}
\usepackage{srcltx}
\usepackage{xcolor}

{}
{}

\DeclareMathOperator{\Mat}{{Mat}}
\DeclareMathOperator{\size}{{size}}
\def\P{{\mathbf{P}}}
\def\e{{\mathbf{e}}}
\def\NP{{\mathbf{NP}}}
\def\WP{{\mathbf{WP}}}
\def\SSP{{\mathbf{SSP}}}
\def\ZOE{{\mathbf{ZOE}}}

\def\SMP{{\mathbf{SMP}}}
\def\BSMP{{\mathbf{BSMP}}}
\def\BKP{{\mathbf{BKP}}}
\def\KP{{\mathbf{KP}}}
\def\IKP{{\mathbf{IKP}}}

\def\GWP{{\mathbf{GWP}}}

\def\KOP{{\mathbf{KOP}}}
\def\SSOP{{\mathbf{SSOP}}}
\def\SMOP{{\mathbf{SMOP}}}
\def\BSMOP{{\mathbf{BSMOP}}}
\def\BGWP{{\mathbf{BGWP}}}

\title{Knapsack problems in groups}

\author[]{Alexei Myasnikov, Andrey Nikolaev, and Alexander Ushakov}
\address{Stevens Institute of Technology, Hoboken, NJ, 07030 USA}
\email{amiasnik,anikolae,aushakov@stevens.edu}

\thanks{The work of the first and third author was partially supported by NSF grant DMS-0914773.}

\begin{document}

\maketitle

\begin{abstract}
We generalize the classical knapsack  and subset sum problems
to arbitrary groups and study the computational complexity of these new problems.
We show that these problems, as well as the bounded submonoid membership problem,   are $\P$-time decidable in hyperbolic groups
and give various examples of finitely presented groups where the subset sum problem is $\NP$-complete.


\end{abstract}

\tableofcontents

\section{Introduction}

\subsection{Motivation}

This is the first in a series of papers on non-commutative discrete (combinatorial) optimization.
In this series  we propose to study complexity of the
classical discrete optimization (DO) problems
in their most  general form, --- in non-commutative groups.  For example, DO problems   concerning integers
(subset sum, knapsack problem, etc.) make perfect sense  when the group
of additive integers is replaced by an arbitrary (non-commutative) group $G$.
The classical lattice  problems are about subgroups (integer lattices) of the
additive groups $\mathbb{Z}^n$ or $\mathbb{Q}^n$, their non-commutative
versions  deal  with arbitrary finitely generated subgroups of a group $G$.
The travelling salesman problem or the Steiner tree problem make
sense for arbitrary finite subsets of  vertices in a given
Cayley graph of a non-commutative infinite group (with the natural graph metric).
The Post correspondence problem carries over in a straightforward fashion
from a free monoid to an  arbitrary group.  This list of examples can be easily extended,
but the point here is that many classical DO problems have natural and interesting non-commutative versions.

The  purpose of this research is threefold.
Firstly, this extends the area of DO to a new and mostly unknown territory, shedding some light on the nature of the problems and facilitating  a deeper understanding of them. In particular, we want to clarify the ``algebraic meaning'' of these problems in the non-commutative situation. Secondly, these are algorithmic problems which are very interesting from the computational algebra viewpoint. They unify various techniques in group theory which seem to be far apart now.  On the practical level,  non-commutative DO problems occur in many  everyday computations in algebra, so it is crucial to study their computational complexity and improve the algorithms.
Thirdly, we aim to develop a robust collection of basic algebraic problems which would serve as building blocks for complexity theory  in non-commutative algebra.
Recall, that the success of the classical complexity theory
in the area of $\NP$ computation is, mostly, due to
a vast collection of discrete optimization problems which are  known  to  be  in $\P$ or $\NP$-complete.
It took many years, starting from the pioneering works of Cook, Levin and Karp in 1970's, to gradually accumulate  this very concrete knowledge.
Nowadays, it is usually a matter of technique to reduce
a new algorithmic problem to some known discrete optimization problem.
This makes the theory of $\NP$-complete DO problems, indeed,  very robust.
In the computational non-commutative algebra
the data base of the known  $\NP$-complete problems is rather small,
and  complexity of some very basic problems is unknown.
Our goal is to start building such a collection in  non-commutative algebra.

\subsection{Stating the problems} \label{sec:problems}

In this paper we focus mostly on subset sum, knapsack, and submonoid membership problems and their variations (described below)  in a given group $G$ generated by a finite or countably infinite set  $X \subseteq G$. We refer to all such problems as {\it knapsack-type} problems in groups.

Elements in $G$ are given as words over the alphabet  $X \cup X^{-1}$.
We begin with three principal decision problems.

\medskip
\noindent{\bf The subset sum  problem ($\SSP$):}
Given $g_1,\ldots,g_k,g\in G$ decide if
  \begin{equation} \label{eq:SSP-def}
  g = g_1^{\varepsilon_1} \ldots g_k^{\varepsilon_k}
  \end{equation}
for some $\varepsilon_1,\ldots,\varepsilon_k \in \{0,1\}$.

\medskip
\noindent
{\bf The knapsack problem ($\KP$):}  Given $g_1,\ldots,g_k,g\in G$
decide if
\begin{equation}\label{eq:IKP-def}
g =_G g_1^{\varepsilon_1} \ldots g_k^{\varepsilon_k}
\end{equation}
for some  non-negative integers $\varepsilon_1,\ldots,\varepsilon_k$.

\medskip

There is also a variation of this problem, termed {\it integer knapsack problem} ($\IKP$), when the coefficients  $\varepsilon_i$ are arbitrary integers. However, it is easy to see that $\IKP$ is $\P$-time reducible to $\KP$ for any group $G$ (see Section   \ref{sec:general_properties}).

The third problem is equivalent to $\KP$ in the classical (abelian) case, but in general it is a completely different problem that is of prime  interest in algebra:

\medskip
\noindent
{\bf Submonoid membership problem ($\mathbf{SMP}$)}:  Given elements $g_1,\ldots,g_k,g\in G$
decide if $g$ belongs to the submonoid generated by $g_1, \ldots, g_k$ in $G$, i.e., if the following equality holds for some $g_{i_1}, \ldots, g_{i_s} \in \{g_1, \ldots, g_k\}, s \in \mathbb{N}$:
\begin{equation}\label{eq:SMP-def}
g = g_{i_1}, \ldots, g_{i_s}.
\end{equation}

\medskip
The restriction of $\SMP$ to the case when the set of generators $\{g_1, \ldots,g_n\}$ is closed under inversion (so the submonoid is actually a subgroup of $G$) is  a well-known problem in group theory, called the {\em generalized word problem} ($\GWP$) or the {\em uniform subgroup membership problem} in $G$. There is a huge bibliography on this subject, we mention some related results  in Section \ref{sec:results}.

As usual in complexity theory, it makes sense to consider the {\em bounded} versions of $\KP$ and $\SMP$, at least they are always decidable in groups where the word problem is decidable. In this case the problem is to verify  if the corresponding equalities (\ref{eq:IKP-def}) and  (\ref{eq:SMP-def}) hold for a given $g$ provided that the number of factors in these equalities  is bounded by a natural number $m$ which is given in the unary form, i.e., as the word $1^m$. In particular, the bounded knapsack problem ($\BKP$) for a group $G$ asks to decide, when  given $g_1,\ldots,g_k,g\in G$ and $1^m\in\MN$, if the equality (\ref{eq:IKP-def}) holds for some
$\varepsilon_i \in \{0,1, \ldots, m \}$.
This problem  is  $\P$-time equivalent to  $\SSP$ in $G$ (see Section \ref{sec:general-properties}), so it suffices for our purposes to consider only $\SSP$ in groups.
On the other hand,  the bounded $\SMP$ in $G$ is very interesting in its own right.

\medskip \noindent
{\bf Bounded submonoid membership problem ($\BSMP$) for $G$:}
Given $g_1, \ldots g_k, g \in G$ and $1^m \in \mathbb{N}$ (in unary)
decide if $g$ is equal in $G$ to a product of the form
$g=g_{i_1}\cdots g_{i_s}$, where $g_{i_1}, \ldots, g_{i_s} \in \{g_1, \ldots, g_k\}$ and  $s\le m$.

\medskip
There are also  interesting and important {\it search} variations of the decision problems above,
when the task is to find an actual solution to equations (\ref{eq:SSP-def}), (\ref{eq:IKP-def}), or (\ref{eq:SMP-def}), provided that some solution  exists (see Section \ref{sec:general_properties} for more details on this). In most cases we solve both the decision and search variations of the problems above simultaneously, while establishing the time complexity upper bounds for the algorithms.  However, as in the classical case, perhaps the most interesting variations  of the search problems are their {\it optimization} versions. It seems these problems were never formally stated before for groups, so we discuss them in a bit more detail here, leaving a more thorough discussion for Section \ref{sec:general_properties}.

\medskip
\noindent
{\bf The subset-sum optimization problem ($\SSOP$) for $G$:}  Given an instance $g_1,\ldots,g_k,g\in G$ of  $\SSP(G)$ find a solution, if it exists,
$\varepsilon_1,\ldots,\varepsilon_k \in \{0,1\}$  subject to the optimization  condition that
the sum $\sum_i \varepsilon_i$ is minimal. Otherwise, output $No \ solutions$.

\medskip
\noindent
{\bf The knapsack optimization problem  ($\KOP$) for $G$:} solve the equation (\ref{eq:IKP-def}) 
with the minimum possible number of factors.

\medskip
In fact, in Section \ref{sec:general_properties} we also discuss other variations of $\KOP$ in groups,  which are even more direct generalization of the classical $\KOP$. In this case  when given $g_1,\ldots,g_k,g\in G$ one has to find $\varepsilon_1,\ldots,\varepsilon_k \in \mathbb{N}$ for which  the product $g_1^{\varepsilon_1}\ldots g_k^{\varepsilon_k}$ is as close to $g$ (in the metric of the Cayley graph of $G$) as possible.

\medskip
\noindent
{\bf The submonoid membership optimization problem ($\SMOP$) for $G$:} %Let $c:G \rightarrow \MN$ be a fixed function.
Given $g_1,\ldots,g_k,g\in G$, % and $c(g_1),\ldots,c(g_k)$, 
express (if possible) $g$ as a product
\begin{equation}\label{eq:MSP_G}
g =_G g_{i_1} \ldots g_{i_m}
\end{equation}
with the  minimum number of factors $m$.

\medskip
The submonoid membership optimization problem plays an important part in geometric group theory.  Indeed, in geometric language it asks to find a geodesic of a given element in a group (relative to a fixed finite generating set) or the distortion of a given element in a given finitely generated subgroup --- both are crucial geometric tasks.

Sometimes (like in hyperbolic groups) the time complexity of the search $\SMP$ is not bounded from above by any computable function, in this case it makes sense to consider the optimization  version of the bounded $\SMP$, called $\BSMOP$, in which one has to solve $\BSMP(G)$ with the minimal possible number of factors.

%%%%%%%%%%

The formal description  of the problems above depends on the given finite (or sometimes countable) generating set $X$ of the group $G$.
To this end if $\Pi$ is any of the algorithmic problems above then by $\Pi(G,X)$ we denote this problem for the group $G$ relative to the generating set $X$.  It
is not hard to show (see Section \ref{sec:general_properties}) that for a given such $\Pi$, provided it is not an optimization problem, replacing one  finite
set of generators of $G$ by another  one ends up in a problem that is $\P$-time reducible
to the initial one.
Therefore, the time complexity of such  problems does not depend on a finite generating set,
it is an invariant of the group $G$.
In view of this, we omit $X$ from the notation   and denote such problems  by $\Pi(G)$.

The typical groups we are interested in here are free,  hyperbolic, abelian, nilpotent, or metabelian.  In all these groups, and this is important, the word problem is decidable in $\P$-time. We might also be interested in constructing some exotic examples of groups where the problems mentioned above have unexpected complexity.

\subsection{What is new?}

The general group-theoretic view on subset-sum and knapsack problems provides several insights.
It is well-known that the classical $\SSP$ is pseudo-polynomial, i.e., it is in $\P$ when the integers
are given in unary form, and it is $\NP$-complete if the integers are given in binaries.
In the group-theoretic framework the classical case occurs when the group $G$ is the additive
group of integers $\mathbb{Z}$. In this case  the complexity of $\SSP(\mathbb{Z})$ depends
whether the set $X$ of generators of $\mathbb{Z}$ is finite or infinite.  Indeed, if $X = \{ 1\}$
then we get $\SSP$ in $\mathbb{Z}$ in the unary form, so in this case it is in $\P$
(likewise for any other finite generating set). However,  if $X =\{2^n \mid n \in \mathbb{N}\}$
then $\SSP(\mathbb{Z})$ is $\P$-time equivalent to the classical $\SSP$
in the binary form, so $\SSP(\mathbb{Z})$  relative to this $X$ is $\NP$-complete (see Example~\ref{ex:inf_gen_set} for details).
To our surprise the situation is quite different (and much more complex) in non-commutative groups.
In non-commutative setting inputs are usually given as  words 
in a fixed generating set of the  group $G$, i.e., in the {\em unary form}
(so the size of the word $x^{2^{10}}$ is $2^{10}$).  It turns out that
in the unary form $\SSP(G)$ is $\NP$-complete even in some very simple
nonabelian groups, such as the metabelian Baumslag--Solitar groups $B(1,p)$, $p\geq 2$,
or the wreath product $\mathbb{Z} \wr \mathbb{Z}$. Furthermore, the reasons why $\SSP(G)$
is hard for such groups $G$ are absolutely different. Indeed, $\SSP$  is
hard for $G = B(1,p)$ because $B(1,p)$ contains exponentially distorted infinite
cyclic subgroups $\mathbb{Z}$; while $\SSP$  is hard for $\mathbb{Z} \wr \mathbb{Z}$
since this group   (also being finitely generated)   contains
an  infinite direct sum $\mathbb{Z}^\omega$.  On the other hand,
$\SSP(G)$ and $\KP(G)$ in  the decision, search, or optimization variations are in $\P$ for
hyperbolic groups $G$ (relative to arbitrary finite generating sets). Observe, that
hyperbolic groups may contain highly (say exponentially)  distorted finitely
generated subgroups, though such subgroups are not abelian. In this case the main reason
why $\SSP(G)$ and  $\KP(G)$ are easy lies in the geometry of hyperbolic groups,
which is asymptotically ``tree-like''.  Another unexpected result which comes from
the polynomial time solution of $\KP$ in hyperbolic groups is that there is a
hyperbolic group $G$ with a finitely generated subgroup $H$ such that the
bounded membership subgroup problem for $H$ is in $\P$, but the standard
subgroup membership problem for $H$ is undecidable. This is the first
result of this sort in groups. Further yet, there are $\P$-time algorithms solving $\SSP$ and $\SMP$ (and all their variations)  in finitely generated nilpotent groups, though in this case the algorithms explore the polynomial growth of such groups, not their geometry.
It remains to be seen if there is a  unifying view-point on  why $\SSP$,   $\KP$, or $\SMP$ could
be hard in a finitely generated group with polynomial time decidable word problem. However,
it is already clear that the nature of the complexity of these problems is  much deeper
than it reveals itself in the commutative case.

\subsection{Results} \label{sec:results}

The subset sum problem is one of the few very basic $\NP$-complete
problems, so it was studied intensely (see \cite{Kellerer-Pferschy-Pisinger:2004}).
Beyond the general interest  $\SSP$ attracted a lot of attention
when Merkle and Hellmann designed a new public
key cryptosystem \cite{Merkle-Hellman:1978} based on some variation of $\SSP$.
The system  was broken by Shamir in \cite{Shamir:1984},
but the interest persists and the ideas survive in
numerous  new cryptosystems and their variations (see \cite{Odlyzko:1990}).
Generalizations of knapsack-type cryptosystems to non-commutative groups
seem quite promising from the viewpoint of post-quantum cryptography,
but even the basic facts on complexity of $\SSP$ and $\KP$ in groups are lacking.

In Section \ref{se:HardSSP} we show that $\SSP(G)$ is $\NP$-complete
in many well-known groups which otherwise are usually viewed as computationally tame,
e.g., free metabelian  groups of finite rank $r \geq 2$, the wreath product $\MZ\wr\MZ$,
or, more generally,  wreath products of any two finitely generated infinite abelian groups. These groups are finitely generated, but not finitely presented.  Even more surprisingly,   $\SSP(G)$ is $\NP$-complete in each of the   Baumslag--Solitar  metabelian groups $B(1,p), p \geq 2$, as well as in the metabelian group
    $$GB = \gp{a,s,t \mid [a,a^t]=1,~ [s,t]=1,~ a^s=aa^t},$$
introduced by Baumslag in \cite{Baumslag:1972}. Notice, that
these groups are finitely presented and have very simple algebraic structure.
Furthermore, it is not hard to see that $\SSP(G)$ is $\NP$-hard if it is $\NP$-hard in some finitely generated subgroup of $G$. In particular, every group containing subgroups isomorphic to any of the groups mentioned above has $\NP$-hard $\SSP$.
Baumslag \cite{Baumslag:1973} and Remeslennikov \cite{Remeslennikov:1973}  showed that every finitely generated metabelian group embeds as a subgroup into a finitely presented metabelian group.
This gives a method to construct various finitely presented groups with $\NP$-complete $\SSP$.  On the other hand, Theorem \ref{th:nilpotent} shows that $\SSP(G)$ is in $\P$ for every finitely generated nilpotent group $G$. The proof is short, but it is based on a rather deep fact that such groups have polynomial growth. One of the main results of the paper is Theorem \ref{th:SSP-hyp} which states that $\SSP(G)$, as well as its search variation, is in $\P$ for any hyperbolic group $G$.  As we mentioned above this also gives a $\P$-time solution to the bounded knapsack problem in hyperbolic groups. In Sections~\ref{se:hyp-ssop} and \ref{se:hyp_kop} we give polynomial time solutions to the various subset sum optimization problems in hyperbolic groups, notably $\SSOP1$ and $\SSOP2$ (see Section~\ref{se:formulation} for definitions).

%In Section~\ref{se:ssp_hyp} we give polynomial time solution to the subset sum problem in hyperbolic groups (Theorem~\ref{th:SSP-hyp}), the bounded submonoid membership problem (Theorem~\ref{th:BSMP}), the subset sum optimization problem (Theorem~\ref{th:hyp-SSOP}), and some of the related optimization problems (Section~\ref{se:hyp-ssop}).

The knapsack problems in groups, especially in their search variations,
are    related to the algorithmic aspects of  the {\em big powers method},
which appeared  long  before  any complexity considerations
(see, for example, \cite{Baumslag:1962}).
Recently, the method shaped up as a basic tool in the study of equations
in free or hyperbolic groups \cite{Bulitko:1970,Kharlampovich_Myasnikov:1998(1), KharlampovichMyasnikovLyutikova:1999, Olshanskii:1989}, algebraic geometry over groups \cite{Baumslag_Miasnikov_Remeslennikov:1999},
completions and group actions \cite{Miasnikov_Remeslennikov:1996, Baumslag_Miasnikov_Remeslennikov:2002,
Miasnikov_Remeslennikov_Serbin:2005}, and became a routine in the theory of hyperbolic
groups (in the form of various lemmas on quasideodesics).  We prove (Theorem \ref{th:IPK}) that $\KP(G)$ together with its search variation are in $\P$ for any hyperbolic group $G$.
 To show  this we reduce $\KP(G)$ in $\P$-time to $\BKP(G)$ in a hyperbolic group $G$. More precisely, we obtain the following result (Theorem \ref{th:PTime_Bulitko}), which is of independent interest.
For any hyperbolic group $G$ there is a polynomial $p(x)$ such that if an equation $g = g^{\varepsilon_1} \ldots g_k^{\varepsilon_k}$ has a solution $\varepsilon_1, \ldots, \varepsilon_k \in \mathbb{N}$, then this equation has a solution with $\varepsilon_i$ bounded by $p(n)$, where $n = |g_1| +\ldots + |g_k| + |g|$ (and it can be found in $\P$-time).   On the other hand,  decidability of quadratic equations in free groups is $\NP$-complete
\cite{Kharlampovich-Lysenok-Myasnikov-Touikan:2010}. We also show that all the optimization versions ($\KOP, \KOP1,\KOP2$) of the knapsack problem in hyperbolic groups are in $\P$. Top solve knapsack problems in hyperbolic groups we developed a new graph technique, which we believe is  of independent interest. Namely, given an instance of a problem we construct a finite labelled digraph (whose  size is polynomial in the size of the instance), such that one can see, just by looking at the graph,  whether or not  a solution to the given instance exists in the group, and if so then find it.

We would like to mention one more result (Theorem \ref{th:BSMP}) here which came as a surprise to us, it states that $\BSMP(G)$ is $\P$-time decidable for every hyperbolic group $G$.
There are hyperbolic groups where the subgroup membership problem
is undecidable even for a fixed finitely generated subgroup, see \cite{Rips:1982}.
It seems this is the first natural example of an undecidable algorithmic problem in groups,
whose bounded version is in $\P$.  It would be interesting to exploit  this direction a bit further.
The famous Mikhailova's construction \cite{Mihailova} shows that the generalized word problem ($\GWP$) is undecidable in the direct product $F \times F$ of a free non-abelian group $F$ with itself. We prove in Section \ref{se:Mikh} (Theorem \ref{th:MikhNP})  that there is a finitely generated subgroup $H$ in $F_2\times F_2$ such that the $\BSMP$ for this fixed subgroup  $H$ in $F_2\times F_2$ is $\NP$-complete. It follows that $\BSMP(G)$ is $\NP$-hard for any group $G$ containing $F_2\times F_2$ as a subgroup.  Notice, that Venkatesan and Rajagopalan proved in \cite{Venkatesan-Rajagopalan:1992}
that in the multiplicative monoid
$\Mat(n,\mathbb{Z})$ of all $n\times n$ integer matrices with $n\geq 20$ the
$\BSMP$ is average-case $\NP$-complete. One of the reasons of this is that $\Mat(20,\mathbb{Z})$ contains a subgroup $F_2 \times F_2$ .

 In another direction observe that  fully residually free (or limit) groups, as well as finitely generated groups acting freely on $\MZ^n$-trees,  have decidable $\GWP$ \cite{Miasnikov_Remeslennikov_Serbin:2006,Nikolaev:thesis,Nikolaev_Serbin:2011(2)}, though the time
 complexity of the decision algorithms is unknown. 
It would be remarkable if $\BSMP$ for such groups was in $\P$. Notice, that Schupp gave a remarkable construction to solve $\GWP$ in $\P$-time  in orientable surface groups, as well as in some Coxeter groups \cite{Schupp:2003}.

We note in passing that the subgroup and submonoid membership problems in a given group could be quite different. For example, Romanovskii proved in \cite{Romanovskii} that $\GWP$ is decidable in every finitely generated metabelian group, but recent examples by Lohrey and Steinberg show that in a free metabelian non-abelian group there is a finitely generated submonoid with undecidable membership problem \cite{LohreySteinberg:2011}.   It would be very interesting to see what is the time complexity of the $\BSMP$ in free metabelian or free solvable groups. Notice that Umirbaev showed in \cite{Umirbaev:1997}
that $\GWP$ in free solvable groups of class $\geq 3$ is undecidable.

\section{General properties}\label{sec:general_properties}

\subsection{Problems set-up}

Since the knapsack-type problems were not previously studied in non-commutative setting
it is worthwhile to say a few words on how we present the data,
models of computations, size functions,  etc.
(we refer to the book \cite{MSU_book:2011} for more details).
Our model of computation is RAM (random access machines).

To make the statements of the problems (from Section \ref{sec:problems})
a bit more precise consider the following. If a  generating set $X = \{x_1,\ldots,x_n\}$ of a
group $G$  is finite, then the {\em size}
of the word $g = x_1\ldots x_k$ is its length $|g|=k$ and the size of
the tuple $g_1, \ldots, g_k, g$ from $G$
is the total sum of the lengths $|g_1| + \ldots +|g_k| +|g|$.

If the generating set $X$ of  $G$ is infinite,
then the size of a letter $x \in X$ is not necessarily equal to $1$,
it depends on how we represent elements of $X$.
In what follows we always assume that there is an efficient injective function $\nu: X \to \{0,1\}^*$ which encodes the elements in $X$
such that for every $u \in \{0,1\}^*$ one can algorithmically recognize if $u \in \nu(X)$, or not.
In this case for $x\in X$ define:
    $$\size(x) = |\nu(x)|$$
and for a word $w  = x_1 \ldots x_n$ with $x_i\in X$ define:
    $$\size(w) = \size(x_1) + \ldots + \size(x_n).$$
Similar to the above the size of a tuple $(g_1, \ldots, g_k, g)$ is:
    $$\size(g_1, \ldots, g_k, g) = \size(g_1) + \ldots + \size(g_k) +\size(g).$$
One can go a bit further and identify elements $x \in X$ with their images $\nu(x) \in \{0,1\}^*$,
and words $w  = x_1 \ldots x_n \in X^*$ with the words $\nu(x_1) \ldots \nu(x_n) \in \{0,1\}^*$.
This gives a homomorphism of monoids $\nu^*: X^* \to \{0,1\}^*$.
If in addition $\nu$ is such that for any $x,y \in X$
the word $\nu(x)$ is not a prefix of $\nu(y)$ (this is easy to arrange),
then:
\begin{itemize}
    \item
$\nu^*$ is injective,
    \item
$\nu^*(X^*)$ and $\nu^*(X)$ are algorithmically recognizable in $\{0,1\}^*$,
    \item
and for every word $v \in \nu^*(X^*)$ one can find the word $w \in X^*$ such that $\nu^*(w) = v$.
\end{itemize}
From now on we always assume that a generating set comes equipped with a function~$\nu$, termed {\it encoding},  satisfying
all the properties mentioned above.
In fact, almost  always all our generating sets $X$ are finite, and in those rare occasions when $X$ is infinite we describe $\nu$ precisely.

In general, we view decision problems as pairs $(I,D)$, where $I$ is the space of instances of
the problem equipped with a  size function $\size: I \to \MN$ and
a set $D \subseteq I$ of affirmative (positive) instances of the problem.
Of course, the set $I$ should be constructible and size function  should be computable.
In all our examples the set $I$ consists either of tuples of words $(g_1, \ldots,g_k,g)$ in
the   alphabet $\Sigma_X$ for some (perhaps, infinite) set of generators $X$ of a group $G$, or,
in the case of  $\BKP$  or $\BSMP$,
tuples of the type  $(g_1, \ldots,g_k,g,1^m)$ where $1^m$ is
a natural number $m$ given in unary.
The problem $(I,D)$ is decidable if there is an algorithm ${\mathcal A}$
that for any $x \in I$ decides whether  $x$ is in $D$ or not (${\mathcal A}$  answers ``Yes'' or ``No'').
The problem is in class $\P$ if there is a decision algorithm ${\mathcal A}$
with polynomial  time function with respect to the size of the instances in $I$, i.e.,
there is a polynomial $p(n)$ such that for any $x \in I$ the algorithm ${\mathcal A}$ starts on $x$,
halts in at most $p(\size(x))$ steps, and gives a correct answer ``Yes'' or ``No''.
Similarly, we define problems in linear  or quadratic time, and non-deterministic polynomial time $\NP$.

Recall that a problem $(I_1,D_1)$ is  {\em $\P$-time  reducible}
to  a problem $(I_2,D_2)$ if there is a $\P$-time computable function
$f:I_1  \to I_2$  such that for any $u \in I_1$ one has $u \in D_1 \Longleftrightarrow f(u) \in D_2$.
Such reductions are usually called either {\em many-to-one} $\P$-time reductions or Karp reductions.
Since we do not use any other reductions we omit ``many-to-one'' from the name and call them $\P$-time reductions.  Similarly, one can introduce linear or quadratic time reductions, etc. We say that two problem are $\P$-time equivalent if each of them $\P$-time reduces to the other.

\subsection{More on the formulation of the problems}
\label{se:formulation}
In this section we continue the discussion from the introduction on  different variations of the problems $\SSP$, $\KP$, $\SMP$ in groups.

There are two ways to state search variations of the problems: the first one, as described  in the introduction, considers only those instances of the problem which are in the ``yes'' part of the problem, i.e., we assume that a solution to the instance exists; the second one is stronger, in this case it is required to solve the decision problem and simultaneously find a solution (if it exists) for a given instance. The former requires only a partial algorithm, while the latter asks for a total one. The weaker version of the problems $\SSP(G)$, $\KP(G)$, $\SMP(G)$ is always decidable in groups $G$  with decidable word problem, while the stronger one might be undecidable (for instance, $\SMP$ in hyperbolic groups). In this paper we consider the stronger version of the search problems.

We mentioned in the introduction that the knapsack optimization problem ($\KOP$) may have different formulations in the non-commutative groups.
Now we explain what we meant.

 Recall first, that perhaps the most typical version of the classical $\KOP$ asks, when given positive integers $a_1, \ldots, a_k,a$ to find $\varepsilon_1, \ldots, \varepsilon_k \in \mathbb{N}$ such that the sum  $ \varepsilon_1a_1 + \ldots +  \varepsilon_ka_k$ is less or equal to $a$ but  maximal possible under this restriction.  One can  generalize this  to  non-commutative groups as follows.

 \medskip \noindent
 {\bf $\KOP1$ for $G$:} Given $g_1,\ldots,g_k,g\in G$
find $\varepsilon_1, \ldots, \varepsilon_k \in \mathbb{N}\cup\{0\}$ with the least
possible distance between $g$ and $g_1^{\varepsilon_1} \ldots g_k^{\varepsilon_k}$ in the Cayley graph $Cay(G,X)$. 

\medskip
 This formulation allows solutions with the ``total weight'' higher than the capacity of the knapsack. To define precisely when a given solution fits in geometrically  in the knapsack we need the following. For elements  $g, h, u \in G$  we say that $u$  {\it belongs to the segment $[g,h]$} if there is a geodesic path in $Cay(G,X)$ from $g$ to $h$ that contains $u$.   Now we can formulate the problem.

  \medskip \noindent
 {\bf $\KOP2$ for $G$:} Given $g_1,\ldots,g_k,g\in G$
find $\varepsilon_1, \ldots, \varepsilon_k \in \mathbb{N}\cup\{0\}$ such that $g_1^{\varepsilon_1} \ldots g_k^{\varepsilon_k}$ belongs to the segment $[1,g]$ and the distance between $g$ and $g_1^{\varepsilon_1} \ldots g_k^{\varepsilon_k}$ in the Cayley graph $Cay(G,X)$ is the least possible.

\medskip

We formulate similar generalizations for the subset sum problem.

 \medskip \noindent
 {\bf $\SSOP1$ for $G$:} Given $g_1,\ldots,g_k,g\in G$
find $\varepsilon_1, \ldots, \varepsilon_k \in \{0,1\}$ such that the distance between $g$ and $g_1^{\varepsilon_1} \ldots g_k^{\varepsilon_k}$ in the Cayley graph $Cay(G,X)$ is the least possible.

   \medskip \noindent
  {\bf $\SSOP2$ for $G$:} Given $g_1,\ldots,g_k,g\in G$
 find $\varepsilon_1, \ldots, \varepsilon_k \in \{0,1\}$ such that the $g_1^{\varepsilon_1} \ldots g_k^{\varepsilon_k}$ belongs to the segment $[1,g]$ and the distance between $g$ and $g_1^{\varepsilon_1} \ldots g_k^{\varepsilon_k}$ in the Cayley graph $Cay(G,X)$ is the least possible.

\medskip

One can also consider optimization problems relative to  a given non-trivial ``weight'' function $c:G \to \mathbb{R}$. For example, instead of optimizing $m\to\min$ in (\ref{eq:MSP_G}), one can ask to optimize $\sum c(g_{i_j})\to\min$. Notice that the optimization problems above correspond to the case when the weight function $c$ is a constant function $c = 1$ on $G$.

\subsection{Examples and basic facts}
\label{sec:general-properties}

The classical ($\SSP$)  is the following algorithmic question.
Given $a_1,\ldots,a_k\in \MZ$ and $M\in \MZ$ decide if
    $$M=\varepsilon_1 a_1+\ldots+\varepsilon_k a_k$$
for some $\varepsilon_1,\ldots,\varepsilon_k \in \{0,1\}$.
It is well known (see \cite{GJ,Papa,Papadimitriou-Steiglitz:1998})
that if the numbers in $\SSP$ are given in  binary, then
the problem is $\NP$-complete, but if they are given in unary,
then the problem is  in $\P$. The examples below show how these
two variations of $\SSP$  appear naturally in the group theory context.

\begin{example}\label{ex:inf_gen_set}
Three variations of subset sum problem for $\MZ$:
\begin{itemize}
    \item
$\SSP(\MZ,\{1\})$
is linear-time equivalent  to the   classical $\SSP$  in which numbers
are  given in unary.
In particular, $\SSP(\MZ,\{1\})$ is  in $\P$.
    \item
For $n\in\MN \cup\{0\}$ put $x_n=2^n$. The set
$X = \{x_n \mid n\in\MN \cup\{0\}\}$
obviously generates $\MZ$.
Fix an encoding $\nu:X^{\pm1} \to \{0,1\}^*$
for $X^{\pm1}$ defined by
$$
\left\{
\begin{array}{rcl}
x_i &\stackrel{\nu}{\mapsto}& 0101(00)^i11,\\
-x_i &\stackrel{\nu}{\mapsto}& 0100(00)^i11.
\end{array}
\right.
$$
Then $\SSP(\MZ,X)$ is  $\P$-time equivalent to its  classical version where the  numbers are given in binary form.
In particular, $\SSP(\MZ,X)$ is   $\NP$-complete.
    \item
Let $X = \{2^n \mid n\in\MN \cup\{0\}\}$ and the number
$2^n$ is represented by the word $01(00)^{2^n}11$ (unary representation).
Then $\SSP(\MZ,X)$ is  in $\P$.
    \item
One can easily define $\SSP$ and $\KP$ in arbitrary algebras $A$ over a field. These problems are equivalent to  $\SSP$ and $\KP$ in the additive group $A^+$ of $A$.
\qed
\end{itemize}
\end{example}

The first  example is of no surprise, of course, since, by definition,  we  treat words representing elements of the group  as in unary.
The second one shows that there might be a huge difference in complexity of $\SSP(G,X)$ for finite and infinite generating sets $X$. The third one indicates that if $X$ is infinite then it really matters how we represent the elements of $X$.

\begin{definition}
Let $G$ and $H$ be groups generated by countable sets $X$ and $Y$ with encodings $\nu$ and $\mu$, respectively. A homomorphism $\varphi:G \to H$
is called $\P$-time computable relative to $(X,\nu), (Y,\mu)$
if there exists an algorithm that
given a word $\nu(u) \in \nu(\Sigma_X^*)$
computes in polynomial time  (in the size of the word $\nu(u)$)  a word $\mu(v) \in \mu(\Sigma_Y^{\ast})$
representing  the element $v=\varphi(u)\in H$.
\end{definition}

\begin{example} \label{ex:finite}
Let $G_i$ be a group generated by a set $X_i$ with encoding $\nu_i$, $i = 1,2$. If $X_1$ is finite then any homomorphism $\varphi:G_1 \to G_2$ is
$\P$-time computable relative to $(X_1,\nu_1), (X_2,\nu_2)$.
\end{example}
To formulate the  following results put
$${\bf  \mathcal P} = \{\SSP,\KP,\SMP,\BKP,\BSMP\}.$$

\begin{lemma}\label{le:SSP_reduction}
Let $G_i$ be a group generated by a set $X_i$ with an encoding $\nu_i$, $i = 1,2$. If $\phi:G_1 \to G_2$ is a $\P$-time computable embedding
relative to $(X_1,\nu_1), (X_2,\nu_2)$ then   ${\bf \Pi}(G_1,X_1)$ is $\P$-time  reducible to ${\bf \Pi}(G_2,X_2)$ for any problem ${\bf \Pi} \in \mathcal{P}$.

\end{lemma}

\begin{proof}
Straightforward.
\end{proof}

In view of Example \ref{ex:finite} we have the following result.

\begin{proposition}\label{le:FiniteX_independence}
If $X$ and $Y$ are finite generating sets for a group $G$, then  ${\bf \Pi}(G,X)$ is $\P$-time  equivalent  to ${\bf \Pi}(G,Y)$ for any problem ${\bf \Pi} \in \mathcal{P}$.
\end{proposition}

\begin{proposition}
Let $G$ be a group and $X$ a generating set for $G$.
Then the  word problem ($\WP$) for $G$ is $\P$-time reducible to ${\bf \Pi}(G,X)$ for any problem ${\bf \Pi} \in \mathcal{P}$.
\end{proposition}

\begin{proof}
Let $w=w(X)$.
Then $w=1$ in $G$ if and only if $1^\varepsilon  = w$ in $G$ for some $\varepsilon \in \{0,1\}$, i.e.,
if and only if the instance $1,w$ of $\SSP(G)$ is positive. Likewise for other problems from ${\mathcal P}$.
\end{proof}

\begin{corollary}
Let $G$ be a group with a generating set $X$. Then:

\begin{enumerate}
\item [1)] $\SSP(G,X)$ (or $\BKP(G,X)$, or $\BSMP(G,X)$) is decidable if and only if the word problem
for $G$ is decidable.
\item [2)] If the word problem for $G$ is $\NP$-hard,
then ${\bf \Pi}(G,X)$ is $\NP$-hard too for any ${\bf \Pi} \in \mathcal{P}$.
\end{enumerate}
\end{corollary}

This corollary shows that from $\SSP$ viewpoint groups with polynomial time decidable word problem are the most interesting.

The following result shows how decision version of $\SSP(G)$ gives a search  algorithm to find an actual sequence of $\varepsilon_i$'s that is a particular solution for a given instance of $\SSP(G)$.

\begin{proposition}
\label{pr:solutions-SSP}
For any group $G$ the search $\SSP(G)$ is $\P$-time Turing reducible to the decision $\SSP(G)$. In particular, if $\SSP(G)$ is in $\P$ then search $\SSP(G)$ is also in $\P$.
\end{proposition}
\begin{proof}
The argument is rather known, so we just give a quick outline to show that it works in the non-commutative case too. Let $w_1,\ldots,w_k,w$ be a given instance of $\SSP(G)$ that has a solution in $G$. To find a solution
$\varepsilon_1,\ldots,\varepsilon_k \in \{0,1\}$
for this instance consider the following algorithm.
\begin{itemize}
    \item
Solve the decision problem for $(w_2,\ldots,w_k),w$ in $G$. If the answer is positive, then
put $\varepsilon_1=0$. Otherwise put $\varepsilon_1=1$ and replace $w$ with $w_1^{-1}w$.
    \item
Continue inductively and find the whole sequence
$\varepsilon_1,\ldots,\varepsilon_k$.
\end{itemize}
\end{proof}

\begin{proposition}
\label{pr:SSP-BKP}
For any group $G$ the following hold:
\begin{itemize}
\item [1)] $\BKP(G)$ is  $\P$-time reducible to $\SSP(G)$;
\item [2)] $\BSMP(G)$, as well as its optimization variation,  is $\P$-time reducible to $\SSOP(G)$.

\end{itemize}
\end{proposition}
\begin{proof}
Given an instance $1^m,w_1,\ldots,w_k,w$ of $\BKP(G)$ we consider a sequence
   \begin{equation} \label{eq:BKP-to-SSP}
   w_1,  \ldots,w_1, w_2, \ldots, w_2,  \ldots, w_k, \ldots,w_k, w \in G,
   \end{equation}
   where each segment $w_i, \ldots, w_i$ has precisely $m$ words $w_i$.
Obviously, the initial instance of $\BKP(G)$ has a solution in $G$ if and only if  $\SSP(G)$ has a solution in $G$ for the  sequence (\ref{eq:BKP-to-SSP}). This establishes a $\P$-time reduction of $\BKP(G)$ to $\SSP(G)$.

To reduce $\BSMP(G)$ to $\SSOP(G)$ for a given instance  $1^m,w_1,\ldots,w_k,w$ of $\BSMP(G)$ consider a sequence
\begin{equation} \label{eq:BSMP-to-SSP}
   w_1,  \ldots, w_k, w_1, \ldots, w_k,  \ldots, w_1, \ldots, w_k, w \in G,
   \end{equation}
where each segment $w_1,  \ldots, w_k$ occurs precisely $m$ times. Obviously, any solution of $\BSMP$ for a given instance gives a solution of $\SSP(G)$ for the sequence (\ref{eq:BSMP-to-SSP}) and vice versa. Hence, solving $\SSOP(G)$ for the sequence (\ref{eq:BSMP-to-SSP}) also solves $\BSMP(G)$  and  $\BSMOP(G)$ for the initial instance. This gives a  polynomial time reduction of $\BSMP(G)$ and $\BSMOP(G)$  to $\SSOP(G)$.

Finally, note that replacing $w_1,\ldots, w_k$ with
$w_1,w_1^{-1},\ldots, w_k,w_k^{-1}$ gives a polynomial time reduction of $\IKP(G)$ to $\KP(G)$.

\end{proof}

\section{Nilpotent groups}

In this section we study the knapsack-type problems in nilpotent groups.

Let $G$ be a  group generated by a finite set $X$. We assume that $X$ is closed under inversion in $G$, so $X^{-1} = X$. For $n \in \mathbb{N}$ we denote by $B_n(X)$ the ball of radius $n$ in the Cayley graph $Cay(G,X)$ of $G$ relative to $X$. We view $B_n(X)$ as a finite directed $X$-labelled graph, which is the subgraph of $Cay(G,X)$ induced by all vertices at distance at most $n$ from the based vertex $1$.

The following result is known as a folklore.
\begin{proposition} \label{le:ball}
Let $G$ be a virtually  nilpotent group generated by a finite set $X$. Then there is $\P$-time algorithms that for a given $n \in \mathbb{N}$ outputs the graph $B_n(X)$.
\end{proposition}
\begin{proof}
Denote by $V_n$ the set of vertices of $B_n(X)$. Clearly,  $V_0 = \{1\}$, and
\begin{equation} \label{eq:growth}
V_n = V_{n-1} \cup_{y \in X}V_{n-1}y.
\end{equation}
 By
%Gromov's theorem \cite{Gromov_pgrowth:1981}
a theorem of Wolf~\cite{Wolf} the growth of $G$ is polynomial, i.e., $|V_i| \leq p(i)$ for some polynomial $p(n)$. It follows from (\ref{eq:growth}) that it takes at most $|X|$ steps (one for each $y \in X$) to construct $B_n(X)$ if given $B_{n-1}(X)$, where each step requires to take an arbitrary vertex $v \in B_{n-1}(X) - B_{n-2}(X)$ (given by some word in $X$), multiply it by the given $y \in X$, and check if the new word $vy$  is equal  or not to any of the previously  constructed vertices.   Recall that finitely generated virtually nilpotent groups are linear, therefore their   word problems are  decidable in polynomial time (in fact, real time~\cite{Holt-Rees:2001}). This shows that $B_n(X)$ can be constructed in a time polynomial in $n$ for a given fixed $G$ and $X$.
\end{proof}

\begin{remark}
The argument above and the following Theorem~\ref{th:nilpotent} are based on the fact that finitely generated virtually nilpotent groups have polynomial growth. By Gromov's theorem \cite{Gromov_pgrowth:1981}  the converse is also true, i.e., polynomial growth implies virtual nilpotence, so the argument cannot be  applied to other classes of groups.
\end{remark}

\begin{theorem} \label{th:nilpotent}
Let $G$ be a finitely generated virtually nilpotent group.
Then $\SSP(G)$ and $\BSMP(G)$, as well as their search and optimization variations, are in $\P$.
\end{theorem}

\begin{proof}
Consider an arbitrary instance $g_1,\ldots,g_k,g$ of $\SSP(G)$.
For every $i=0,\ldots,k$ define a set
    $$P_i=\{ g_1^{\varepsilon_1} \ldots g_i^{\varepsilon_i} \mid \varepsilon_1,\ldots,\varepsilon_i \in \{0,1\}\}.$$
Clearly, the given instance is positive if and only if $g\in P_k$.
The set $P_i$ can be constructed recursively using the formula:
\begin{equation}\label{eq:formula_P}
P_i = P_{i-1} \cup P_{i-1} \cdot g_i.
\end{equation}
Observe that all elements of $P_k$ lie in the ball $B_m(X)$, where $m = |g_1| + \ldots + |g_k|$. Using formula (\ref{eq:formula_P}) one can in polynomial time identify all vertices in $B_m(X)$ that belong to $P_k$ (an argument  similar to the one in Proposition~\ref{le:ball} works here as well). During the identification process one can also in polynomial time for each vertex $v \in P_k$ associate  a tuple $(\varepsilon_1, \ldots, \varepsilon_k)$, where $\varepsilon_i \in \{0,1\}$, such that $v = g_1^{\varepsilon_1} \dots g_k^{\varepsilon_k}$  in $G$ and with minimal possible total sum $\varepsilon_1 + \ldots + \varepsilon_k$.
To do this one needs only to keep the best current tuple during the identification process for each already identified vertex in $P_k$.  Now if  the element $g$ is given as a  word $w$ in $X$, one can trace $w$ off in the graph $B_m(X)$ and check if this word defines an element from $P_k$ or not. If it does, one can get an optimal solution from the tuple associated with the vertex in $P_k$ defined by $w$. This solves $\SSP$ and $\SSOP$ in $G$ in polynomial time. By Proposition \ref{pr:SSP-BKP} this  implies that $\BSMP(G)$ and  $\BSMOP(G)$ are in $\P$ as well.
\end{proof}

%\subsection{Subset sum optimization problems}

\section{Groups with hard $\SSP$}
\label{se:HardSSP}

In this section we give many examples of various finitely generated and
finitely presented groups $G$ with $\NP$-hard $\SSP(G)$.
We start with an infinitely generated group $\MZ^\omega$, a direct sum of countably many copies of the infinite cyclic group $\MZ$.
We view elements of $\MZ^\omega$ as sequences $\MN \to \MZ$ with finite support.
For $i\in\MN$ by $\e_i$ we denote a sequence such that $\e_i(j) = \delta_{i,j} $,
where $\delta_{i,j}$ is the Kronecker delta function. The set $E = \{\e_i\}_{i\in\MN}$
is a basis for $\MZ^\omega$. We fix an encoding $\nu:E^{\pm1} \to \{0,1\}^*$
for the generating set $E$ defined by:
$$
\left\{
\begin{array}{rcl}
\e_i & \stackrel{\nu}{\mapsto}& 0101(00)^i11,\\
-\e_i &\stackrel{\nu}{\mapsto}& 0100(00)^i11.
\end{array}
\right.
$$

\begin{proposition}
$\SSP(\MZ^\omega,E)$ is $\NP$-complete.
\end{proposition}

\begin{proof}
Below we reduce a problem known to be $\NP$-complete, namely zero-one equation problem,
to $\SSP(\MZ^\omega,E)$.
Recall that a vector $v \in \MZ^n$ is called a {\em zero-one} vector
if each entry in $v$ is either $0$ or $1$.
Similarly, a square matrix  $A\in \Mat(n,\MZ)$ is called a {\em zero-one}
matrix if each entry in $A$ is either $0$ or $1$.
Denote by $1_n$ the vector $(1,\ldots,1) \in \MZ^n$.
The following problem is $\NP$-complete  (see \cite{Dasgupta-Papadimitriou-Vazirani:2006}).

\medskip\noindent
{\bf Zero-one equation problem (ZOE):}
Given a zero-one matrix $A \in \Mat(n,\MZ)$ decide
if there exists a zero-one vector $\ovx \in \MZ^n$ satisfying
$A\cdot \ovx = 1_n$, or not.

\medskip
$\ZOE$ can be reduced to $\SSP(\MZ^\omega,E)$ as follows.
Given zero-one $n\times n$ matrix $A = (a_{ij})$ compute elements
    $$g_i=\sum_{j=1}^n a_{ij} \e_j \in \MZ^\omega\ \ \ (\mbox{for } i=1,\ldots,n)$$
and put $g = \e_1+\ldots+\e_n$. Clearly, the tuple $g_1,\ldots,g_n,g$ is $\P$-time computable
and $A$ is a positive instance of $\ZOE$
if and only if $g_1,\ldots,g_n,g$ is a positive instance
of $\SSP(\MZ^\omega,E)$.
This establishes a $\P$-time reduction of $\ZOE$ to
$\SSP(\MZ^\omega,\{\e_i\})$, as claimed.
\end{proof}

The next proposition is obvious.

\begin{proposition}\label{pr:SPP_complete_criterion}
Let $G$ be a group generated by a  set $X$. If $\varphi:\MZ^\omega \rightarrow G$ is  a $\P$-time computable embedding relative to the generating sets $E$ and $X$
then $\SSP(G)$ is $\NP$-hard. If, in addition, the word problem for $G$
is decidable in polynomial time, then $\SSP(G)$ is $\NP$-complete.
\qed
\end{proposition}

This result gives a wide class of groups $G$  with $\NP$-hard
or $\NP$-complete $\SSP(G)$.

\begin{proposition}\label{pr:SSP-hard}
The following groups have $\NP$-complete $\SSP$:
\begin{itemize}
\item[(a)] Free metabelian  non-abelian groups of finite rank.
\item[(b)] Wreath product $\MZ\wr\MZ$.
\item[(c)] Wreath product of two finitely generated infinite abelian groups.
\end{itemize}
\end{proposition}

\begin{proof}
Let $M_n$ be a free metabelian group with basis $X = \{x_1, \ldots,x_n\}$, where $n \geq 2$.
It is not hard to see that the elements
$$e_i = x_1^{-i}[x_2,x_1]x_1^i\ \ \  (\mbox{for } i\in\MN)$$
freely generate a free abelian group $\MZ^\omega$ (see, for example,
the description of normal forms of elements of $M_n$  in \cite{Bryant-Roman'kov:1999}).
This gives a $\P$-time computable embedding of $\MZ^\omega$ into $M_n$
relative to the generating sets $E$ and $X$. It is known that
the word problem in finitely generated metabelian groups is in $\P$ (see, for example,  \cite{Miasnikov_Romankov_Ushakov_Vershik:2010}).
Hence, by Proposition \ref{pr:SPP_complete_criterion},
$\SSP(M_n)$ is $\NP$-complete and (a) holds.

The wreath product of two infinite cyclic groups generated by $a$ and $t$ respectively
is a finitely generated infinitely presented group
    $$G = \gpr{a,t}{[a,t^{-i}at^i]=1,\ (\mbox{for } i\in\MN)}.$$
The set $\{t^{-i}at^i \mid i \in \MN \}$ freely generates a subgroup isomorphic to $\MZ^\omega$.
In fact, the map $\e_i \to t^{-i}at^i$ defines a $\P$-time computable embedding of $\MZ^\omega$
into $G$ relative to the generating sets $E$ and $\{a,t\}$. Proposition \ref{pr:SPP_complete_criterion}
finishes the proof of (b).

Finally, consider arbitrary infinite finitely generated abelian groups $A$ and $B$.
Then $A \simeq A_1 \times \MZ$ and $B = B_1 \times \MZ$ and $\MZ \wr \MZ$
can be $\P$-time embedded into $A \wr B$. The result now follows from (b).
\end{proof}

Thompson's group $F$ has a finite presentation
    $$\gpr{a,b}{[ab^{-1}, a^{-1}ba]=1,\ [ab^{-1}, a^{-2}ba^2]=1}.$$
It is a remarkable group due to a collection of very unusual properties
that made it a counterexample to many general conjectures in group theory
(see \cite{CFP}).

\begin{proposition}
The subset sum problem for the Thompson's group
$F$ is $\NP$-complete.
\end{proposition}

\begin{proof}
According to \cite{Cleary:2005} the wreath product $\MZ\wr\MZ$
can be embedded into $F$ with no distortion.
The word problem for $F$ is decidable in polynomial time \cite{CFP,SU1}.
Now the result follows from Propositions \ref{pr:SPP_complete_criterion} and \ref{pr:SSP-hard}.
\end{proof}

In \cite{Baumslag:1972} Baumslag gave an example of a finitely
presented metabelian group
\begin{equation}\label{eq:BaumslagGroup}
GB = \gp{a,s,t \mid [a,a^t]=1,~ [s,t]=1,~ a^s=aa^t}.
\end{equation}

\begin{proposition}
$\SSP(GB)$ is $\NP$-complete.
\end{proposition}

\begin{proof}
As shown in \cite{Baumslag:1972} the subgroup $\gp{a,t}$ of the group $GB$
is isomorphic to $\MZ\wr\MZ$.
Hence, $\MZ\wr\MZ$ embeds into $GB$ and since $\MZ \wr \MZ$
is finitely generated this embedding is $\P$-time computable.
The word problem for $GB$ is in $\P$ because $GB$ is a finitely presented metabelian group.
Thus, by Propositions \ref{pr:SSP-hard} and \ref{pr:SPP_complete_criterion},
$\SSP(GB)$ is $\NP$-complete.
\end{proof}

There are many examples of finitely presented metabelian groups with $\NP$-complete
subset sum problem. Indeed, Baumslag~\cite{Baumslag:1973} and Remeslennikov~\cite{Remeslennikov:1973}  proved
that every finitely generated metabelian group $G$ embeds into a finitely presented metabelian group $G^*$.
Since $G$ is finitely generated this embedding is $\P$-time computable
with respect to the given finite generating sets.
Therefore, if $G$ contains $\P$-time computably embedded subgroup
$\MZ^\omega$ so does $G^*$.

Now we describe another type of examples of finitely presented groups $G$
with $\NP$-complete $\SSP(G)$.
Consider the well-known Baumslag--Solitar metabelian group
    $$BS(m,n) = \langle a,t \mid t^{-1}a^mt = a^n\rangle.$$

\begin{theorem}
$\SSP(BS(1,2))$ is $\NP$-complete.
\end{theorem}

\begin{proof}
We showed in Example \ref{ex:inf_gen_set} that $\SSP(\MZ,X)$ is $\NP$-complete
for a generating set $X = \{x_n=2^n \mid n\in\MN \cup\{0\}\}$. The map
    $$x_n \to t^{-n}at^n$$
is obviously $\P$-time computable and defines an embedding $\phi:\MZ \to BS(1,2)$
because $t^{-n}at^n = a^{2^n}$.
Hence, $\SSP(\MZ,X)$ $\P$-time reduces to $\SSP(BS(1,2))$.
Thus, $\SSP(BS(1,2))$ is $\NP$-complete.
\end{proof}

In fact, it is easy to prove that $\SSP(BS(m,n))$ is $\NP$-complete
whenever $|m|\ne |n|$ and $m,n\ne 0$.
It is less obvious that $\SSP(BS(n,\pm n))$ is in $\P$. We shortly outline the algorithm here.
Here we use graphs defined in the next section (see Figure \ref{fi:SSP_graph})
in which edges are allowed to be labeled with arbitrary powers of $a$.
Start with the graph $\Gamma(w_1,\ldots,w_k,w)$.
Repeatedly apply Britton's lemma to the graph: 
\begin{itemize}
\item
for any path $s_1 \stackrel{t^{\pm 1}}{\rightarrow} s_2 \stackrel{a^{cm}}{\rightarrow} s_3 \stackrel{t^{\mp 1}}{\rightarrow} s_4$ add the edge
$s_1 \stackrel{a^{cm}}{\rightarrow} s_4$ in case of $BS(n,n)$, or the edge $s_1 \stackrel{a^{-cm}}{\rightarrow} s_4$ in case of $BS(n,-n)$ (where $c\in\MZ$), and
\item
for any path $s_1 \stackrel{a^{s}}{\rightarrow} s_2 \stackrel{a^{t}}{\rightarrow} s_3$ add the edge
$s_1 \stackrel{a^{s+t}}{\rightarrow} s_3$.
\end{itemize}
The procedure terminates in polynomial time because powers $m$ are bounded by the length of the input.
The answer is ``Yes'' if there exists an $\varepsilon$-edge from $\alpha$ to $\omega$.

\begin{corollary}
If a group $G$ contains a subgroup isomorphic to $B(m,n)$
with $|m|\ne |n|$ and $m,n\ne 0$, then $\SSP(G)$ is $\NP$-hard.
\qed
\end{corollary}

\section{$\SSP$ in hyperbolic groups}\label{se:ssp_hyp}

In this section we prove that the subset sum problem
is $\P$-time decidable for every hyperbolic group.
We refer to \cite{Gromov_hyperbolic, Alonso-etal:1991} for introduction to hyperbolic groups.
The proofs in this section are based on some results from \cite{Ushakov:thesis,MU1}
(see also the book \cite{MSU_book:2011}).

Let $G= \gpr{X}{R}$ be a finitely presented group.
A word $w=w(X)$ is called trivial, or a relator, or null-homotopic in $G$  if $w=_G\varepsilon$.
A {\em van Kampen diagram} over the presentation $\gpr{X}{R}$ is a planar finite cell complex $D$
given with a specific embedding $D\subseteq \MR^2$ satisfying the following conditions.
\begin{itemize}
    \item
$D$ is connected and simply connected.
    \item
Each edge is labeled with a letter $x\in X$.
    \item
Some vertex $v\in \partial D$ is specified as a base-vertex.
    \item
Each cell is labeled with a word from $R$.
\end{itemize}

\begin{theorem*}[van Kampen lemma]
A word $w=w(X)$ represents the identity of $G$ if and only if
there exists a van Kampen diagram with the boundary label $w$.
\end{theorem*}

For a diagram $D$ one can define a {\em dual graph} $Dual(D) = (V,E)$, where
the vertex set $V$ is the set of all cells of $D$ (including the outer cell)
and the edge set $E$ is the set of all pairs of cells $(c_1,c_2)$ in $D$ sharing at
least one vertex. The maximal distance in $Dual(D)$ from the outer cell
to another cells is called the {\em depth} of $D$, denoted by $\delta(D)$.
By the depth $\delta(w)$ of a trivial in $G$ word $w$  we understand
the minimal depth of a van Kampen diagram with the boundary label $w$
(see \cite{MU1,MSU_book:2011}).

\begin{proposition}[\cite{MU1}]\label{pr:hyperbolic_depth}
Let $G$ be a hyperbolic group given by a finite presentation $G=\gpr{X}{R}$. Then for any
word  $w=w(X)$ with $w =_G 1$ one has  $\delta(w) = O(\log_2 |w|).$
\end{proposition}

\subsection{Finite state automata over hyperbolic groups}\label{se:automata}
Our polynomial time solution for the subset sum problem
for hyperbolic groups uses finite state automata and two
operations,  called {\em $R$-completion} and {\em folding},
described below.

{\bf Notation.} For a finite automaton $\Gamma$ over the alphabet $X$
we denote by $L(\Gamma)$ the set of all words accepted by $\Gamma$.
By $|\Gamma|$ we denote the number of states in $\Gamma$.
In general, for a set $S\subset X^\ast$ by $\overline{S}$ we denote
the  image of $S$ in $G = \gpr{X}{R}$ under the standard  epimorphism $X^\ast \to G$.

\subsubsection{$R$-completion}\label{se:R_completion}

Recall that a group presentation $\gpr{X}{R}$
is called {\em symmetrized} if $R=R^{-1}$ and $R$
is closed under taking cyclic permutations of its elements.
Given a symmetrized presentation $\gpr{X}{R}$ and
an automaton $\Gamma$ over $\Sigma_X = X \cup X^{-1}  \cup \{\varepsilon\}$ one can construct a new automaton
$\CC(\Gamma)$ obtained from $\Gamma$ by adding a loop labeled by $r$ for every $r\in R$
at every state $v\in \Gamma$.
By $R$-completion of $\Gamma$ we understand the graph  $\CC^{k}(\Gamma)$ for some $k\in\MN$.
We want to point out that unlike in \cite{MU1},
we do not perform Stallings' foldings after adding relator-loops.
Instead, we perform a special transformation of the automaton
described in Section \ref{se:Folding}.

\begin{proposition}[Properties of $\CC(\Gamma)$]\label{pr:completion_properties}
For every $\gpr{X}{R}$ and $\Gamma$ the following holds:
\begin{itemize}
    \item[(a)]
$\Gamma$ is a subgraph of $\CC(\Gamma)$.
    \item[(b)]
$\overline{L(\Gamma)} = \overline{L(\CC(\Gamma))}$.
    \item[(c)]
$|\CC(\Gamma)| \le |\Gamma| \cdot \|R\|$, where $\|R\| = \sum_{r\in R} |r|$.
\end{itemize}
\end{proposition}

\begin{proof}
Follows from the construction of $\CC(\Gamma)$.
\end{proof}

\subsubsection{Non-Stallings folding}
\label{se:Folding}

Given an automaton $\Gamma$ over a group alphabet $\Sigma_X$
one can construct a new automaton $\mathcal F(\Gamma)$ obtained from $\Gamma$
by a sequence of steps, at each step adding new edges as described below.  For every pair of consecutive edges
of the  form  shown in the left column of the table below
we add the edge from the right column of the table (in the same row), provided this edge is not yet  in the graph.
\[
\begin{tabular}{|l|l|}
\hline
$s_1\stackrel{x}{\rightarrow} s_2 \stackrel{x^{-1}}{\rightarrow} s_3$ & $s_1\stackrel{\varepsilon}{\rightarrow} s_3$\\
\hline
$s_1\stackrel{x}{\rightarrow} s_2 \stackrel{\varepsilon}{\rightarrow} s_3$ & $s_1\stackrel{x}{\rightarrow} s_3$ \\
\hline
$s_1\stackrel{\varepsilon}{\rightarrow} s_2 \stackrel{x}{\rightarrow} s_3$ & $s_1\stackrel{x}{\rightarrow} s_3$ \\
\hline
$s_1\stackrel{\varepsilon}{\rightarrow} s_2 \stackrel{\varepsilon}{\rightarrow} s_3$ & $s_1\stackrel{\varepsilon}{\rightarrow} s_3$.\\
\hline
\end{tabular}
\]
Clearly, the procedure eventually stops, because the number of vertices does not increase
and the alphabet $X$ is finite.

\begin{proposition}\label{pr:folding_properties}
$\overline{L(\Gamma)} = \overline{L(\mathcal F(\Gamma))}$
for any finite automaton $\Gamma$ over the alphabet $\Sigma_X$.
\end{proposition}

\begin{proof}
The language as a set of reduced words does not change.
\end{proof}

\begin{lemma}\label{le:goodloop_existence}
Let $\langle X\mid R\rangle$ be a finite presentation of a hyperbolic group. Let $\Gamma$ be an acyclic automaton over $\Sigma_X$ with at most $l$ nontrivially labeled edges. Then $1\in \overline{L(\Gamma)}$ if and only if
there exists $u\in L(\CC^{O(\log l)}(\Gamma))$
%where $l=|w|+\sum|w_i|$,
satisfying $u=_{F(X)}\varepsilon$.
\end{lemma}

\begin{proof}
If $1\in \overline{L(\Gamma)}$, then there exists $v\in L(\Gamma)$
such that $v=_G1$. The length of $v$ is bounded by $l$.
By Proposition \ref{pr:hyperbolic_depth} the depth of $v$ is bounded
by $O(\log |v|)$. Let $D$ be a diagram with perimeter label $v$ of depth $O(\log |v|)$.

Next, we mimic the proof of \cite[Proposition 16.3.14]{MSU_book:2011}.
Cut $D$ to obtain a new ``forest'' diagram $E$ of the height $l$
with a perimeter label $v u$, where $u=_{F(X)}\varepsilon$.
See \cite[Figure 16.2]{MSU_book:2011}.
The diagram $E$ embeds into $\CC^{O(\log l)}(\Gamma)$
and the initial segment $v$ of the perimeter label of $E$
is mapped onto the corresponding word in $\Gamma$.
This way we obtain a path from $\alpha$ to $\omega$
labeled with $u$, as claimed.

The other direction of the statement follows from Proposition~\ref{pr:completion_properties}. %and~\ref{pr:folding_properties}.
\end{proof}

\begin{proposition}\label{pr:goodloop_existence}
Let $\langle X\mid R\rangle$ be a finite presentation of a hyperbolic group. Let $\Gamma$ be an acyclic automaton over $\Sigma_X$ with at most $l$ nontrivially labeled edges. Then $1\in \overline{L(\Gamma)}$ if and only if
$\mathcal F(\CC^{O(\log l)}(\Gamma))$ contains an edge $\alpha \stackrel{\varepsilon}{\rightarrow} \omega$.
\end{proposition}
\begin{proof}
Follows from Lemma~\ref{le:goodloop_existence}, definition of $\mathcal F$ and Proposition~\ref{pr:folding_properties}.
\end{proof}

\subsection{The algorithm}
\label{se:Algorithm}

For a sequence of words $w_1,\ldots,w_k,w$ construct an automaton
$\Gamma=\Gamma(w_1,\ldots,w_k,w)$ as in Figure \ref{fi:SSP_graph}.
\begin{figure}[h]
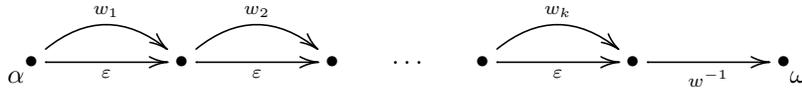

\centerline{
\xygraph{
!{<0cm,0cm>;<2cm,0cm>:<0cm,2cm>::}
!{(0,0)}*+{\bullet}="0"
!{(1,0)}*+{\bullet}="1"
!{(2,0)}*+{\bullet}="2"
!{(3,0)}*+{\bullet}="3"
!{(4,0)}*+{\bullet}="4"
!{(5,0)}*+{\bullet}="5"
"0":@[|(1.5)]@/^0.5cm/"1"^{w_1}
"0":@[|(1.5)]"1"_{\varepsilon}
"1":@[|(1.5)]@/^0.5cm/"2"^{w_2}
"1":@[|(1.5)]"2"_{\varepsilon}
"3":@[|(1.5)]@/^0.5cm/"4"^{w_k}
"3":@[|(1.5)]"4"_{\varepsilon}
"4":@[|(1.5)]"5"_{w^{-1}}
!{(2.5,0)}*+{\ldots}="6"
!{(-.1,-.1)}*+{\alpha}
!{(5.1,-.1)}*+{\omega}
}}
\caption{\label{fi:SSP_graph}The graph $\Gamma(w_1,\ldots,w_k,w)$.}
\end{figure}
Edges labeled with $w_i$'s on the picture are sequences of edges
labeled with the letters involved in $w_i$'s.
The initial state is the leftmost state $\alpha$ and the final state
is the rightmost state $\omega$.
It is easy to see that $|L(\Gamma)| \le 2^n$ and the length of every $u\in L(\Gamma)$
is bounded by $|w|+\sum|w_i|$. The lemma below is obvious.

\begin{lemma}\label{le:SSP_graph}
Let $\Gamma=\Gamma(w_1,\ldots,w_k,w)$.
An instance $w_1,\ldots,w_k,w$ of $\SSP(G)$ is positive
if and only if $1\in \overline{L(\Gamma)}$.
\qed
\end{lemma}

\begin{theorem}\label{th:SSP_hyperbolic}
Let $G$ be a hyperbolic group given by a finite presentation $\gpr{X}{R}$
and $w_1,\ldots,w_k,w\in F(X)$.
Then $w =_G w_1^{\varepsilon_1} \ldots w_k^{\varepsilon_k}$
for some $\varepsilon_1,\ldots,\varepsilon_k \in \{0,1\}$ if and only if
the graph $\mathcal F(\CC^{O(\log (|w|+\sum|w_i|))}(\Gamma))$ contains
the edge $\alpha \stackrel{\varepsilon}{\rightarrow} \omega$.
\end{theorem}

\begin{proof}
Follows immediately from the Lemma~\ref{le:SSP_graph} above and Proposition~\ref{pr:goodloop_existence} with $l=|w|+\sum|w_i|$.
%\ref{pr:completion_properties} and \ref{pr:folding_properties}
%that $\overline{L(\Gamma)} = \overline{L(\mathcal F(\CC^{O(\log (|w|+\sum|w_i|))}(\Gamma)))}$.
%Therefore, sufficiency holds. Necessity follows from Lemma \ref{le:goodloop_existence}
%and the definition of the folding procedure.
\end{proof}

As a corollary we get the following principal result.
\begin{theorem} \label{th:SSP-hyp}
 $\SSP(G) \in \P$ for any  hyperbolic group $G$.
\end{theorem}

\begin{corollary}
The search variation of $\SSP(G)$ is  $\P$-time  solvable for any hyperbolic group $G$.
\end{corollary}
\begin{proof}
Follows from Proposition \ref{pr:solutions-SSP}.
\end{proof}

Another corollary concerns with the bounded knapsack problem.

\begin{corollary}\label{co:BKP}
 $\BKP(G) \in \P$ for any  hyperbolic group $G$.
\end{corollary}

\begin{proof}
Follows from Proposition \ref{pr:SSP-BKP}
\end{proof}

\subsection{The bounded submonoid membership problem}\label{sec:bounded_bsmp}

In this section we consider the bounded submonoid problem ($\BSMP$) in hyperbolic groups.

\begin{theorem}\label{th:BSMP}
Let $G$ be a hyperbolic group. Then $\BSMP(G) \in \P$.
\end{theorem}

\begin{proof}
The proof uses the technique introduced in Section~\ref{se:automata}.
% is analogous to the proof of Theorem \ref{th:SSP_hyperbolic}.
Let $w_i$ be a word representing the element $g_i$, $ i = 1, \ldots,k$,
and $w$ a word representing $h$.
We construct a finite graph $\Gamma$ similar to the one considered
in Section \ref{se:Algorithm} as shown in Figure~\ref{fig:BSMP}.
\begin{figure}[h]
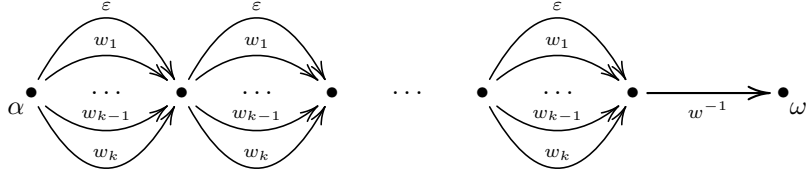

\centerline{
\xygraph{
!{<0cm,0cm>;<2cm,0cm>:<0cm,2cm>::}
!{(0,0)}*+{\bullet}="0"
!{(1,0)}*+{\bullet}="1"
!{(2,0)}*+{\bullet}="2"
!{(3,0)}*+{\bullet}="3"
!{(4,0)}*+{\bullet}="4"
!{(5,0)}*+{\bullet}="5"
"0":@[|(1.5)]@/^1cm/"1"^{\varepsilon}
"0":@[|(1.5)]@/^0.5cm/"1"^{w_1}
!{(.5,0)}*+{\ldots}
"0":@[|(1.5)]@/^-0.5cm/"1"^{w_{k-1}}
"0":@[|(1.5)]@/^-1cm/"1"^{w_k}
"1":@[|(1.5)]@/^1cm/"2"^{\varepsilon}
"1":@[|(1.5)]@/^0.5cm/"2"^{w_1}
!{(1.5,0)}*+{\ldots}
"1":@[|(1.5)]@/^-0.5cm/"2"^{w_{k-1}}
"1":@[|(1.5)]@/^-1cm/"2"^{w_k}
"3":@[|(1.5)]@/^1cm/"4"^{\varepsilon}
"3":@[|(1.5)]@/^0.5cm/"4"^{w_1}
!{(3.5,0)}*+{\ldots}
"3":@[|(1.5)]@/^-0.5cm/"4"^{w_{k-1}}
"3":@[|(1.5)]@/^-1cm/"4"^{w_k}
"4":@[|(2)]"5"_{w^{-1}}
!{(2.5,0)}*+{\ldots}
!{(-.1,-.1)}*+{\alpha}
!{(5.1,-.1)}*+{\omega}
}}
\caption{\label{fig:BSMP}Graph $\Gamma$.}
\end{figure}

%An argument similar to the one in Theorem \ref{th:SSP_hyperbolic} shows that
By Proposition~\ref{pr:goodloop_existence},
$g_1,\ldots,g_k,h,1^m$
is a positive instance of $\BSMP(G)$  if and only if
the graph $F(\CC^{O(\log (|w|+m\sum|w_i|))}(\Gamma))$ contains
the edge $\alpha \stackrel{\varepsilon}{\rightarrow} \omega$.
Hence the result.
\end{proof}

\subsection{Optimization problems}\label{se:hyp-ssop}

In this section we solve in polynomial time several optimization problems in hyperbolic groups.

\begin{theorem} \label{th:hyp-SSOP}
Let $G$ be a hyperbolic group. Then the subset sum optimization problem in $G$ is in $\P$.
\end{theorem}

\begin{proof}
Let $w_1,\ldots,w_k,w$ be a given instance of $\SSOP(G)$. We may assume that $w_i\ne \varepsilon$.
Our algorithm is very similar to the algorithm described in Section \ref{se:Algorithm}, one needs only to use one extra decoration of the  graph $\Gamma=\Gamma(w_1,\ldots,w_k,w)$ from Figure \ref{fi:SSP_graph} (and all of its completions and foldings). We equip the graph $\Gamma$ with a function
$\gamma:E(\Gamma) \rightarrow \MN \cup\{0\}$, termed the {\em price function}.
This  function $\gamma$ is equal to  zero on all edges except for the last edge in each word $w_i$, where  it is equal to $1$. The price of a path $p=e_1,\ldots,e_m$ in $\Gamma$ from a vertex $\alpha$ to a vertex $\beta$
is defined by
$$
\gamma(p)=\sum \gamma(e_i).
$$
Now we describe how $\gamma$ changes under completions and foldings. In the completion process, as described  in Section \ref{se:R_completion},
one  adds loops labelled with  relations $r\in R$ and assigns zero  price ($\gamma = 0$) to every new edge. Under the folding process every new edge gets the price value that is equal to the sum of the prices of the folded edges.
When folding two edges $e_1$ and $e_2$, the folding algorithm adds a new edge $e$  only if such an edge does not already exist in the graph. If such $e$ is there already  we replace its $\gamma$ value $\gamma(e)$ with the minimum of $\gamma(e)$ and $\gamma(e_1) + \gamma(e_2)$.

Now we construct the graph $\Delta = \mathcal F(\CC^{O(\log (|w|+\sum|w_i|))}(\Gamma))$ together with the price function $\gamma$, it takes only polynomial time in the size of the instance $w_1,\ldots,w_k,w$.
By Theorem \ref{th:SSP_hyperbolic} $\SSP(G)$  has an affirmative solution
if and only if the graph $\Delta$ contains the edge $\alpha \stackrel{\varepsilon}{\rightarrow} \omega$.
Furthermore, it is not hard to see that the price of this  edge is the minimal number of $w_i$'s required in
the expression (\ref{eq:MSP_G}). Now, it is straightforward to find the actual
optimal solution  from the graph $\Delta$.
\end{proof}

\begin{corollary}
Let $G$ be a hyperbolic group. Then the bounded submonoid optimization problem is in $\P$.
\qed
\end{corollary}
\begin{proof}
Follows from the theorem above and Proposition \ref{pr:SSP-BKP}.
\end{proof}
%The crucial property that allows polynomial time solutions for this and the other
%problems in hyperbolic groups is that potential solutions have bounded length
%and the set of all potential solutions can be represented as a language accepted
%by a relatively small automaton.
We would like to point out that the usual unbounded subgroup  membership problem is undecidable in some hyperbolic groups (Rips~\cite{Rips:1982}), hence the search subgroup membership problem in a given hyperbolic group, though decidable, cannot have any computable upper bound on its time complexity.
Nevertheless, in some special cases one can solve the unbounded optimization problem
in polynomial time, e.g., in free groups.

\begin{theorem}\label{ptime_submonoid_length}
The submonoid membership optimization  problem in a free group is polynomial time solvable.
\end{theorem}

\begin{proof}
We construct first a directed graph $\Gamma$ for $\{w_1,\ldots,w_k\}$
with the tail labelled with $w^{-1}$ as in  Fig.~\ref{fig:bouqet_length}.
\begin{figure}[htb]
 \centering
 \includegraphics[height=1.4in]{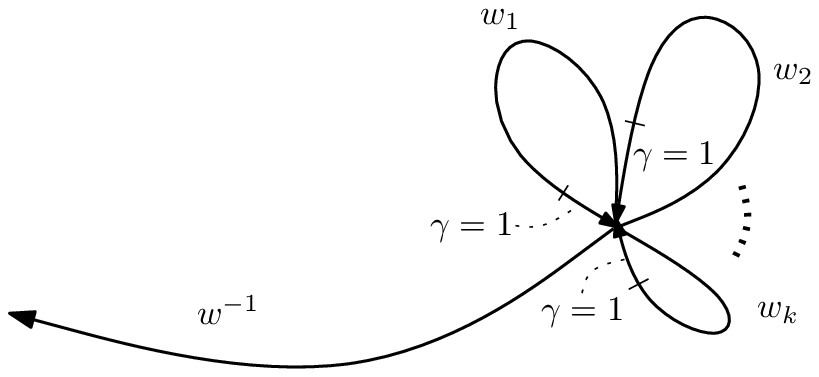}
 \caption{Grapth $\Gamma$, Theorem~\ref{ptime_submonoid_length}.}\label{fig:bouqet_length}
\end{figure}
Then we apply Stallings' foldings, decorated with a price function as in the proof of Theorem \ref{th:hyp-SSOP}.
A few details are in order here. The values of the price function are stored in binary.
It is easy to see that the bit-length of the values of $\gamma$
in the folded graph is bounded by $|X|\cdot \rb{|w|+\sum |w_i|}$.
Hence, all computations can be done in polynomial time.
\end{proof}

We give polynomial time solutions to more optimization problems in hyperbolic groups in Section~\ref{se:hyp_kop}.
%\au{What happens if you add negative numbers? Is it interesting?}
%\begin{itemize}
%\item
%Minimum can be $-\infty$ which of course can not be realized.
%This case should be easy to recognize, I suspect that it happens if
%and only if an $\varepsilon$-loop of negative price appears.
%\item
%\end{itemize}

\section{Knapsack problem in hyperbolic groups}

In this section we study the knapsack problem $\KP(G)$ 
in hyperbolic groups $G$ relative to finite generating sets.
The main goal is to prove the following theorem.

\begin{theorem} \label{th:IPK}
Let $G$ be a hyperbolic group generated by finite set $X$. Then $\KP(G,X) \in \P$.
Moreover, there exists a $\P$-time algorithm which for any positive instance $g_1, \ldots, g_k, g \in G$ of
$\KP(G)$ computes a sequence of non-negative integers $\varepsilon_1,\ldots,\varepsilon_k$ such that $g_1^{\varepsilon_1} \ldots g_k^{\varepsilon_k} = g$ in $G$.
\end{theorem}

To prove this theorem we need some   results in hyperbolic groups.

\subsection{Auxiliary results in hyperbolic groups}
In our notation we follow the paper \cite{Miasnikov_Nikolaev:2011}, where one can also find all the needed notions and definitions.

\begin{lemma}\label{le:qg_fellow}
Let $\mathcal H$ be a $\delta$-hyperbolic geodesic metric space. Let $p,q$ be two $(\lambda,\varepsilon)$-quasigeodesic
paths in $\mathcal H$ joining points $P_1,P_2$ and $Q_1,Q_2$, respectively. Suppose $H\ge 0$ is such that $|P_1Q_1|\le H$ and
$|P_2Q_2|\le H$. Then there exists $K=K(\delta,\lambda,\varepsilon,H)\ge 0$ such that $p,q$ asynchronously $K$-fellow travel.
\end{lemma}
\begin{proof}
This is well-known. For example, see~\cite{Miasnikov_Nikolaev:2011}.
\end{proof}

\begin{lemma}\label{linear-fellow}
Let $\mathcal H$ be a $\delta$-hyperbolic geodesic metric space. Let $p,q$ be two $(\lambda,\varepsilon)$-quasigeodesic
paths in $\mathcal H$ joining points $P_1,P_2$ and $Q_1,Q_2$, respectively. Suppose $H\ge 1$ is such that $|P_1Q_1|\le H$ and
$|P_2Q_2|\le H$. Then there exists $K_1=K(\delta,\lambda,\varepsilon)\ge 0$ such that $p,q$ asynchronously $K_1H$-fellow travel.
\end{lemma}
\begin{proof}
Let $K_1$ be the constant $K_1=K(\delta,\lambda,\varepsilon, 1)$ provided by Lemma~\ref{le:qg_fellow}. Then triangle inequality gives a linear in $H$
bound $K\le K_1H$ on the constant of fellow travel, as shown in Fig.~\ref{fig:fellow}.
\begin{figure}[htb]
 \centering
 \includegraphics[height=1.4in]{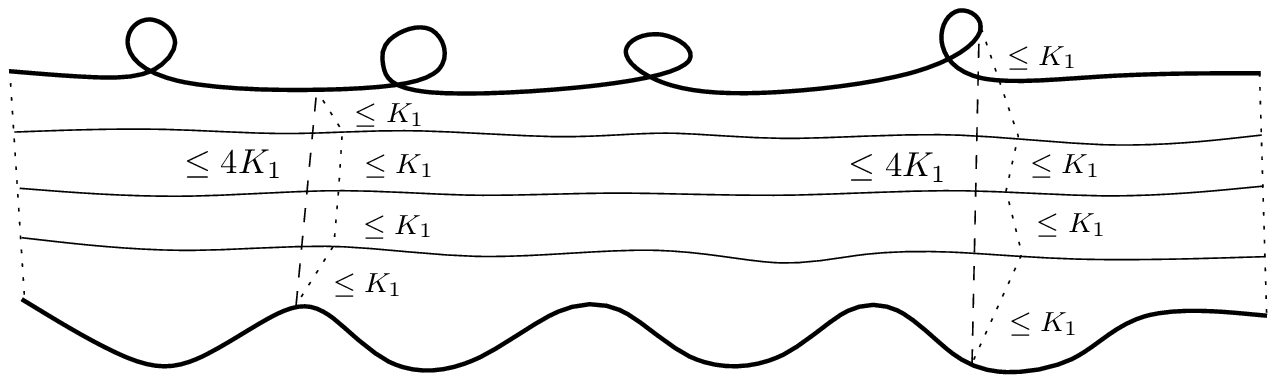}
 \caption{Lemma~\ref{linear-fellow}, $H=4$.}\label{fig:fellow}
\end{figure}
\end{proof}

It is well-known that (quasi-)geodesic polygons in a hyperbolic space are ``thin''. In the following lemma we give a logarithmic bound on ``thickness''
of a quasi-geodesic $m$-gon.

\begin{lemma}\label{ngon} Let $p_1p_2\ldots p_m$ be a $(\lambda,\varepsilon)$-quasigeodesic $m$-gon in a $\delta$-hyperbolic space.
Then there is a constant $H=H(\delta,\lambda,\varepsilon)$ such that each side $p_i$, $1\le i\le m$, belongs to the closed $(H+H\ln m)$-neighborhood
of the union of other sides $p_1,\ldots,p_{i-1},p_{i+1},\ldots,p_m$.
\end{lemma}
\begin{proof}
First we prove the lemma in the case $m=2^l$. Drawing a diagonal in a quadrangle, we obtain a constant $H_1$ such
that every side of a $(\lambda,\varepsilon)$-quasigeodesic quadrangle belongs to the closed $H_1$-neighborhood of the union of other three sides.
(Note that $H_1$ also delivers the same statement for triangles.) Since $H_1\le H_1+H_1\ln 4$, this provides the base case $l=2$.

Suppose the statement is proven for $m=2^l$ with $H=3H_1$. Prove that $H=3H_1$ also suffices in the case $m=2^{l+1}$.

Indeed, let $p_1,\ldots,p_m$ be an $m$-gon. For each $1\le i\le m-1$, let $p_i$ have endpoints $P_i$ and $P_{i+1}$, and $p_m$ have endpoints $P_m$ and $P_1$. Draw geodesic diagonals $q_1,\ldots,q_{2^l}$ so that $q_i$,
$1\le i\le 2^l-1$, joins points $P_{2i-1}$, $P_{2i+1}$, and $q_{2^l}$ joins $P_{m-1}$, $P_1$. (See Fig.~\ref{fig:thin}.)
\begin{figure}[h]
 \centering
 \includegraphics[height=1.75in]{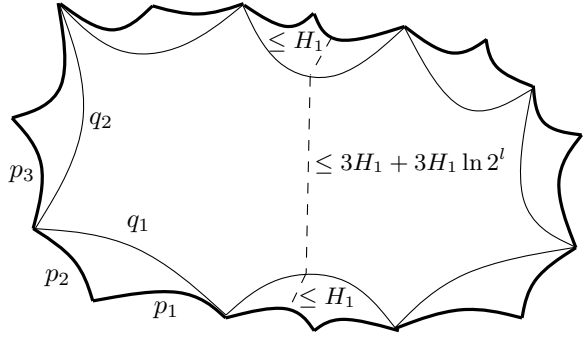}
 \caption{Drawing diagonals in a $2^{3+1}$-gon.}\label{fig:thin}
\end{figure}

Since $(2^l)$-gon $q_1\ldots q_{2^l}$ belongs to the closed $H_1$-neighborhood of $p_1\ldots p_m$, by triangle inequality, every side of
$p_1\ldots p_m$ is contained in the $(H_1+H_1+(3H_1+3H_1\ln 2^l))$-neighborhood of the union of other sides. Since
$$H_1+H_1+3H_1+3H_1\ln 2^l\le
%3H_1+ 3H_1(\ln 2^l+2/3)\le 3H_1+3H_1(\ln 2^l +\ln 2)=
3H_1+3H_1\ln 2^{l+1},$$
the case $m=2^{l+1}$ is obtained.

Finally, for arbitrary $m$, considering an $m$-gon as a degenerate $2^l$-gon, where $2^{l-1}<m\le 2^l$, we
obtain that $H=6H_1$ proves the statement of the lemma.
\end{proof}

Let $\langle X\mid R\rangle$ be a finite presentation of a hyperbolic group $G$. We say that an element $g\in G$ is
{\em cyclically reduced} if it has minimal geodesic length among all elements in the conjugacy class $g^G$. We say that a geodesic word $w$ in
the alphabet $X$ is {\em cyclically reduced} if the corresponding group element $g=\overline{w}$ is cyclically reduced.
We say that two elements $g,h\in G$ are {\em commensurable} if their powers are conjugated, i.e. there exist $m,n\in\mathbb Z$, not both zero, $c\in G$
such that $c^{-1}g^mc=h^n$.

\begin{lemma}\label{cyclic}
For any finite presentation $\langle X\mid R\rangle$ of a hyperbolic group $G$, there exist constants $\lambda,\varepsilon$
with the following property. For any cyclically reduced word $w$, for any $n\in\mathbb Z$, the word $w^n$ is
$(\lambda,\varepsilon)$-quasigeodesic.
\end{lemma}
\begin{proof} See Lemma~27 of~\cite{Olshanskii:1991}.
\end{proof}

\begin{lemma}\label{commensurable} Suppose $\langle X\mid R\rangle$ is a presentation of a group $G$. Let $w_1$ and $v_1$ be words in
the alphabet $X$ such that the corresponding elements $g_1=\overline{w_1}$, $f_1=\overline{v_1}$ of $G$ have infinite order.
Suppose $w$ and $v$ are infinite paths in the Cayley graph of $\langle X\mid R\rangle$
labeled by $w_1^\infty$ and $v_1^\infty$.

Then there exists a constant $L=L(|X|)$ with the following property. If a segment of $w$, containing
at least $|v_1|L^K$ copies of $w_1$, asynchronously $K$-fellow travels with a segment of $v$, then $g_1$ and $f_1$ are commensurable.
\end{lemma}
\begin{proof}
Note that in such a case, the endpoint of each copy of $w_1$ is connected by a path labeled by a word $u_i$ of length at most $K$ with a point on $v$. Therefore,
the endpoint of each copy of $w_1$ is connected by a path labeled $u_id_i$ with an endpoint of a copy of $v_1$, where $d_i$ is a terminal subword of $v_1$ (see Figure~\ref{fig:commensurable}).
\begin{figure}[htb]
 \centering
 \includegraphics[height=0.9in]{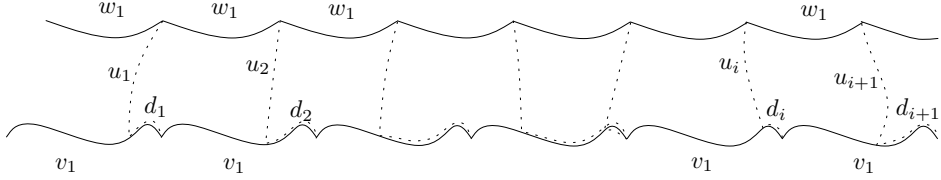}
 \caption{Paths labeled $u_id_i$ connect endpoints of copies of $w_1$ and $v_1$.}\label{fig:commensurable}
\end{figure}
Since
there are at most $|v_1|\cdot (2|X|)^{K}$ words of the form $u_id_i$, taking $L>2|X|$ guarantees that words $u_id_i$ repeat,
yielding that in $G$ one has
\begin{equation}
g_1^{k_1}=\overline{u_id_i} f_1^{k_2} \overline{u_id_i}^{-1},\label{conjugated}
\end{equation} i.e. that $g_1$, $f_1$ are commensurable (for example, see Figure~\ref{fig:comm2}).
\begin{figure}[h]
 \centering
 \includegraphics[height=0.9in]{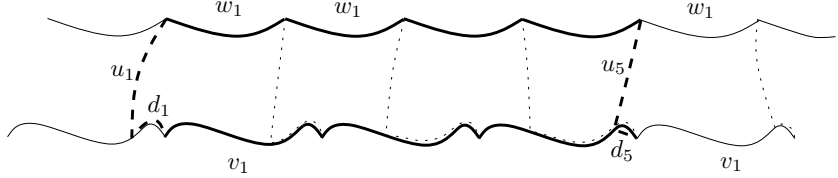}
 \caption{If $u_1d_1=u_5d_5$, then $g_1^{4}=\overline{u_1d_1} f_1^{3} \overline{u_1d_1}^{-1}$.}\label{fig:comm2}
\end{figure}
\end{proof}

\subsection{Reduction of $\KP$ to $\BKP$}\label{se:hyp_KP}

Let $G$ be a hyperbolic group.
The following result, which is of independent interest,
$\P$-time reduces $\KP(G)$ to $\BKP(G)$.
This proves Theorem  \ref{th:IPK} because $\BKP(G)$
is $\P$-time decidable by Corollary \ref{co:BKP}.

\begin{theorem}\label{th:PTime_Bulitko}
Let $G$ be a hyperbolic group. Then there is a polynomial $p(x)$ such that if for $g_1,\ldots,g_k, g\in G$ there exist integers
$\varepsilon_1,\ldots,\varepsilon_k \in \mathbb Z$ such that
$$g = g_1^{\varepsilon_1} \ldots g_k^{\varepsilon_k}$$
then there exist such integers $\varepsilon_1,\ldots,\varepsilon_k \in \mathbb Z$ with
$$\max\{|\varepsilon_1|, \ldots, |\varepsilon_k|\} \leq p(|g_1|+\ldots|g_k|+|g|).$$
\end{theorem}

\begin{proof}
Let $E$ be the maximum order of torsion elements in $G$ (it is well-defined since a hyperbolic group has a finite number of conjugacy classes of finite subgroups, see~\cite{Bogopolskii-Gerasimov:1995} or~\cite{Brady:2000}), or $E=1$ if $G$ is torsion-free. For every torsion element $g_i$, $1\le i\le k$, we may assume
that $|\varepsilon_i|<E$. Suppose now that among $g_1,\ldots, g_k$ there is at least one element of infinite order.

Fix a presentation $\langle X\mid R\rangle$ of $G$ and denote $|g_1|_X+\ldots +|g_k|_X+|g|_X=n$  (here $|\cdot|_X$ denotes the geodesic length
with respect to $X$).

Let $g_{i_1},\ldots, g_{i_m}$ be the entirety of elements of infinite order among $g_1,\ldots,g_k$.
For each infinite order $g_{i_j}$, $1\le j\le m$, let $h_j, c_j$ be such that $g_{i_j}=c_j^{-1}h_jc_j$ and $h_j$ cyclically reduced. Note that
$|h_j|_X, |c_j|_X\le |g_{i_j}|_X\le n$.
Given a product $g_1^{\varepsilon_1} \ldots g_k^{\varepsilon_k}$, denote blocks of powers of finite order elements as follows:
$$c_j g_{i_{j}+1}^{\varepsilon_{i_{j}+1}}\ldots g_{i_{j+1}-1}^{\varepsilon_{i_{j+1}-1}} c^{-1}_{j+1}=b_{j+1}\qquad \mathrm{for}\quad 1\le j\le m-1,$$
$$
b_1=g_1^{\varepsilon_1}\ldots g_{i_{1}-1}^{\varepsilon_{i_1-1}} c^{-1}_{1},\qquad
b_{m+1}=c_m g_{i_{m}+1}^{\varepsilon_{i_{m}+1}}\ldots g_{k}^{\varepsilon_{k}}.
$$
For convenience put $\varepsilon_{i_j}=\alpha_j$ so that
$$
g_1^{\varepsilon_1} \ldots g_k^{\varepsilon_k}=b_1 h_1^{\alpha_1}b_2\ldots b_m h_m^{\alpha_m}b_{m+1}.
$$
Note that $|b_i|\le n\cdot n^E+2n\le 3n^{E+1}$.

Consider and $(2m+2)$-gon with sides $q_1p_1q_2\ldots p_mq_{m+1}r$ where:
\begin{itemize}
\item $q_i,$ $1\le i\le m+1$, is labeled by a geodesic word representing $b_i$,
\item $p_i,$ $1\le i\le m$, is labeled by a $(\lambda,\varepsilon)$-quasigeodesic word representing $h_i^{\alpha_i}$ (according to Lemma~\ref{cyclic}),
\item $r$ is labeled by a geodesic word representing $g$.
\end{itemize}
We will show that given a sufficiently large polynomial bound on $M$, if at least one $|\alpha_j|>M$, then some
powers $|\alpha_i|>M$ can be reduced while preserving the equality
$g=b_1 h_1^{\alpha_1}b_2\ldots b_m h_m^{\alpha_m}b_{m+1}$.

Assume some $|\alpha_j|\ge M$, with $M$ to be chosen later. By Lemma~\ref{ngon}, the side $p_j$ of the polygon belongs to a closed $(H+H\ln (2m+2))$-neighborhood of the union of the other sides, where $H$ only depends on $X$, $R$, $\lambda$ and $\varepsilon$. By
Lemma~\ref{cyclic},  $\lambda$ and $\varepsilon$, in turn, only depend on $X$, $R$.

If two points $p_j(t_1), p_j(t_2), t_1< t_2$ are $(H+H\ln (2m+2))$-close to a side $q$ (where $q$ is one of sides $p_i, q_i, r$), then by
Lemma~~\ref{linear-fellow} the subpath $p_j(t), t_1\le t\le t_2$ asynchronously $K_2=K_1(H+H\ln (2m+2))$-fellow travels with a subpath
of $q$. Therefore we may assume that $p_j$ is split into at most $(2m+1)$ segments, so that
each segment asynchronously $K_2=K_1(H+H\ln (2m+2))$-fellow travels with a segment of another side. By pigeonhole principle,
at least one segment of $p_j$ contains at least
\begin{equation}(M-2m)/(2m+1)\ge \frac{M}{2m+1}-2\ge \frac{M}{3n}-2=M_1\label{M1-2}
\end{equation}
copies of the word representing $h_j$. Denote this segment
of $p_j$ by $p$ and its fellow traveler by $s$.

Note that since $p_j$ is $(\lambda,\varepsilon)$-quasigeodesic, geodesic length of $s$ is at least
\begin{equation}\lambda^{-1}(M_1|h_1|_X-2K_2)-\varepsilon.\label{length-2}\end{equation}
We show below that given sufficiently large lower bound on $M$, $p$ can fellow travel neither with $q_i$, nor with $r$.
Choosing
\begin{equation}M>3n(\lambda(\varepsilon+3n^{E+1})+2K_2+2)=Q_1(n)\label{bound1-2}
\end{equation}
 guarantees $M_1>\lambda(\varepsilon+3n^{E+1})+2K_2$, so by~(\ref{length-2})
geodesic length of $|s|_X>3n^{E+1}$,
which eliminates the possibility that $s$ is a segment of $q_i$, $1\le i\le m+1$. Note that $Q_1(n)$ in~(\ref{bound1-2}) is of degree ${E+2}$ in $n$ since
$K_2=K_1(H+H\ln (2m+2))\le K_1(H+3nH)$. The same bound~(\ref{bound1-2}) also prohibits fellow travel with $r$ since geodesic length of $r$ is at most $n<3n^{E+1}$.

From~(\ref{bound1-2}) we conclude that with
%\begin{equation} M>\left(3n(\lambda(\alpha+2n)+2K_2+2)\right)+\left(3n(L e^{K_1H} n(4n)^{K_1H}+2)\right),\label{1plus2}\end{equation}
\begin{equation} M>Q_1(n)+E,\label{1plus2-2}\end{equation}
the only possibility
is that $p$ fellow travels with a segment of some $p_l$, $l\neq j$.

By Lemma~\ref{commensurable},
there exists $L$ (depending on $X$) such that if $p$ $K_2$-fellow travels with a segment of $p_l$ and $M_1>n L^{K_2}$, then
$h_j$ and $h_l$ are commensurable and form a rectangle $h_j^{k_1}=d^{-1} h_l^{k_2} d$ (see~(\ref{conjugated})) with $k_1$ between $0$ and $\alpha_j$, and $k_2$
between $0$ and $\alpha_l$.
In that case, $\alpha_j$ and $\alpha_l$ can be replaced by $(\alpha_j-k_1)$ and $(\alpha_l-k_2)$, respectively,
preserving the equality $g=b_1h_1^{\alpha_1}\ldots h_m^{\alpha_m}b_{m+1}$. (See Fig.~\ref{remove}.)
\begin{figure}[htb]
 \centering
 \includegraphics[height=1.5in]{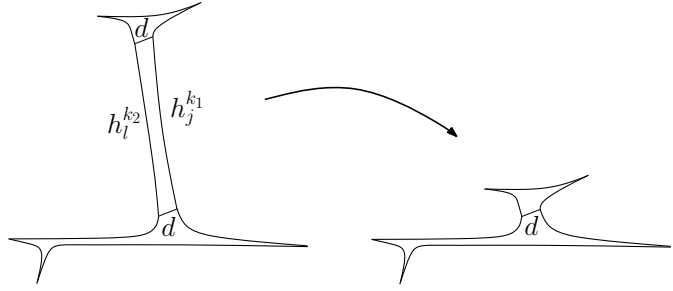}
 \caption{Removing rectangle $h_j^{k_1}=d^{-1} h_l^{k_2} d$.}\label{remove}
\end{figure}
Note that
\begin{align*}
n L^{K_2} & =nL^{K_1(H+H\ln (2m+2))}\\
%L e^{K_1H} n(e^{\ln(2m+2)})^{K_1H}=
%&=ne^{\ln LK_1(H+H\ln (2m+2))}=\\
&= L^{K_1H} n(2m+2)^{K_1H\ln L}\\
&\le L^{K_1H} n(4n)^{K_1H\ln L}.
\end{align*}
Hence, $M_1\ge n L^{K_2}$ is guaranteed by
\begin{equation}M>3n(L^{K_1H} n(4n)^{K_1H\ln L}+2)=Q_2(n),\label{bound2-2}\end{equation}
which is of degree $\le(2+K_1H\ln L)$ in $n$.
Consider
\begin{equation} M=Q_1(n)+Q_2(n)+E.\label{1plus2-3}\end{equation}
that satisfies inequalities~(\ref{bound1-2}),~(\ref{1plus2-2}) and~(\ref{bound2-2}).
By the argument above, if some $|\varepsilon_i|>M$ and $g_i$ is a torsion element, then $\varepsilon_i$ can be replaced with
$\varepsilon'_i$ where $|\varepsilon'_i|<E<M$.
If some $|\varepsilon_i|>M$ and $g_i$ is an infinite order element, then $\varepsilon_i$ and some $\varepsilon_j$ can be
replaced by $\varepsilon'_i$ and some $\varepsilon'_j$, respectively, where
$|\varepsilon'_i|<|\varepsilon_i|$ and $|\varepsilon'_j|<|\varepsilon_j|$.

Repeating this procedure, we eventually obtain that for every $1\le i\le k$, $|\varepsilon_i|<M$. It is only left to note that $M$
in~(\ref{1plus2-3}) is %a polynomial of
of degree $\max\{E+2, 2+K_1H\ln L\}$ in $n$, where $E,K_1,H,L$ depend only on the presentation $\langle X\mid R\rangle$.
This finishes proof of the theorem.
\end{proof}

If infinite order elements $g_1,\ldots, g_m$ are not commensurable,
then a stronger version of Theorem~\ref{th:PTime_Bulitko} holds.

\begin{theorem}\label{strongBP}
Let $\langle X\mid R\rangle$ be a finite presentation of a hyperbolic group $G$. Let $w_1,\ldots, w_m$ be words in the alphabet $X$ of total length $|w_1|+|w_2|+\ldots +|w_m|=n$.
Suppose that the elements $g_1=\overline{w_1},\ldots,g_m=\overline{w_m}$ in $G$ satisfy:
\begin{itemize}
    \item[(a)]
$g_i$ has infinite order for every $1\le i\le m$.
    \item[(b)]
$g_i,g_j$ are commensurable if and only if $i=j$.
\end{itemize}
Then there is a polynomial $p(x)$ that depends solely on the presentation
$\langle X\mid R\rangle$, such that for any $K\in\mathbb N$, if the geodesic length of the product
$$
h=g_1^{\alpha_1}\cdots g_m^{\alpha_m}
$$
is less than $K$, then
$$
\max\{|\alpha_1|,\ldots,|\alpha_m|\}\le K\cdot p(n).
$$
%at least one $|\alpha_i|\ge K\cdot p(n)$, then the geodesic length of $h$ is at least $K$.
\end{theorem}

\begin{proof}
We assume $|h|_X<K$ and come up with a lower bound on $M$ that makes this impossible.

For each $g_i$, $1\le i\le m$, let $h_i, c_i$ be such that $g_i=c_i^{-1}h_ic_i$ and $h_i$ cyclically reduced. Note that
$|h_i|_X, |c_i|_X< |g_i|_X\le n$.
Given numbers $\alpha_1,\ldots,\alpha_m$, consider the product
$$h=g_1^{\alpha_1}g_2^{\alpha_2}\cdots g_m^{\alpha_m}=
c_1^{-1}h_1^{\alpha_1}c_1c_2^{-1}h_2^{\alpha_2}c_2\cdots c_{m-1}c_m^{-1}h_m^{\alpha_m}c_m,$$
and $(2m+2)$-gon with sides $q_1p_1q_2\ldots p_mq_{m+1}r$ where:
\begin{itemize}
\item $q_1$ is labeled by a geodesic word representing $c_1^{-1}$,
\item $p_i,$ $1\le i\le m$, is labeled by a $(\lambda,\varepsilon)$-quasigeodesic word representing $h_i^{\alpha_i}$ (according to Lemma~\ref{cyclic}),
\item $q_i$, $2\le i\le m$, is labeled by a geodesic word representing $c_{i-1}c_{i}^{-1}$,
\item $q_{m+1}$ is labeled by a geodesic word representing $c_{m}$,
\item $r$ is labeled by a geodesic word representing $h$.
\end{itemize}
Assume some $\alpha_j\ge M$, with $M$ to be chosen later. The proof then proceeds similar to proof of the Theorem~\ref{th:PTime_Bulitko}
using Lemmas~\ref{ngon} and~\ref{linear-fellow}.

Analogous to~(\ref{bound1-2}),
\begin{equation}
M>3n(\lambda(\varepsilon+2n+K)+2K_2+2)=P_1(n,K)\label{bound1}
\end{equation}
forbids fellow travel with $q_i$ (since $|q_i|< 2n$) or $r$ (since $|r|<K$). Then
\begin{equation}M>3n(L^{K_1H} n(4n)^{K_1H\ln L}+2)=P_2(n),\label{bound2}\end{equation}
similar to~(\ref{bound2-2}), allows to apply Lemma~\ref{commensurable} if $p_j$ fellow travels with a segment of some $p_l$, $l\neq j$.
Since $g_j, g_l$ are not commensurable, this possibility is also eliminated.

Now notice that
%From~(\ref{bound1}) and~(\ref{bound2}) we conclude that with
%\begin{equation} M>\left(3n(\lambda(\varepsilon+2n)+2K_2+2)\right)+\left(3n(L e^{K_1H} n(4n)^{K_1H}+2)\right),\label{1plus2}\end{equation}
\begin{equation} M=P_1(n,K)+P_2(n)\label{1plus2}\end{equation}
satisfies~(\ref{bound1}) and~(\ref{bound2}), making $|h|_X<K$ impossible.
Observe that $M$ in~(\ref{1plus2}) is %a polynomial of
of degree $\max\{2, 2+K_1H\ln L\}$ in $n$,
where $K_1, H, L$ ultimately depend only on the presentation $\langle X\mid R\rangle$, and linear in $K$.
\end{proof}

\subsection{Knapsack optimization problems}
\label{se:hyp_kop}
In this section we describe polynomial solutions to $\KOP$, $\KOP1$, $\KOP2$, $\SSOP1$, and $\SSOP2$ (see Section~\ref{se:formulation}) for hyperbolic groups.

By Thereom~\ref{th:PTime_Bulitko}, in hyperbolic groups $\KP$ reduces to $\BKP$.
Therefore, Theorem~\ref{th:hyp-SSOP} is enough to give a polynomial time solution to $\KOP$.
Similarly, by Thereom~\ref{th:PTime_Bulitko}, $\KOP1$ in a given hyperbolic group $\P$-time reduces to $\SSOP1$. The following theorem suffices to solve the latter problem in polynomial time.

\begin{theorem}\label{th:hyp_kop1} Let $G$ be a hyperbolic group given by a finite presentation $\langle X\mid R\rangle$.
There exists a polynomial time algorithm that, given $g_1,\ldots,g_k,g\in G$ and a unary $N\in\mathbb N\cup\{0\}$,
finds $\varepsilon_1, \ldots, \varepsilon_k \in\{0,1\}$ such that the distance between $g$ and $g_1^{\varepsilon_1} \ldots g_k^{\varepsilon_k}$ in the Cayley graph $Cay(G,X)$ does not exceed $N$, or outputs {\em ``No solutions''} if no such $\varepsilon_1,\ldots,\varepsilon_k$ exist.
\end{theorem}
\begin{proof}
By a standard argument, it is enough to  solve the corresponding decision problem: given $g_1,\ldots,g_k,g\in G$ and a unary $N\in\mathbb N\cup\{0\}$,
decide whether there exist $\varepsilon_1, \ldots, \varepsilon_k \in \{0,1\}$ such that the distance between $g$ and $g_1^{\varepsilon_1} \ldots g_k^{\varepsilon_k}$ in the Cayley graph $Cay(G,X)$ does not exceed $N$.

We consider graph $\Gamma=\Gamma(w_1,\ldots,w_k,w,N)$ similar to the one in Fig.~\ref{fi:SSP_graph}, accommodating a ball of radius $N$ centered at $w$, as in Fig.~\ref{fig:KOP1}.

\begin{figure}[h]
\centerline{
\xygraph{
!{<0cm,0cm>;<2cm,0cm>:<0cm,2cm>::}
!{(0,0)}*+{\bullet}="0"
!{(1,0)}*+{\bullet}="1"
!{(1.5,0)}*+{\bullet}="3"
!{(2.5,0)}*+{\bullet}="4"
!{(3.2,0)}*+{\bullet}="5"
!{(3.6,0)}*+{\bullet}="6"
!{(4.3,0)}*+{\bullet}="7"
!{(5.3,0)}*+{\bullet}="8"
"0":@[|(1.5)]@/^1cm/"1"^{w_1}
"0":@[|(1.5)]"1"_{\varepsilon}
"3":@[|(1.5)]@/^1cm/"4"^{w_k}
"3":@[|(1.5)]"4"_{\varepsilon}
"4":@[|(1.5)]@/^1.3cm/"5"^{x_1}
!{(2.85,0.48)}*+{\ldots}
"4":@[|(1.5)]@/^0.55cm/"5"^{x_m}
"4":@[|(1.5)]@/^0cm/"5"^{\varepsilon}
"4":@[|(1.5)]@/^-0.55cm/"5"^{x_{1}^{-1}}
!{(2.85,-0.35)}*+{\ldots}
"4":@[|(1.5)]@/^-1.3cm/"5"^{x_m^{-1}}
"6":@[|(1.5)]@/^1.3cm/"7"^{x_1}
!{(3.95,0.48)}*+{\ldots}
"6":@[|(1.5)]@/^0.55cm/"7"^{x_m}
"6":@[|(1.5)]@/^0cm/"7"^{\varepsilon}
"6":@[|(1.5)]@/^-0.55cm/"7"^{x_{1}^{-1}}
!{(3.95,-0.35)}*+{\ldots}
"6":@[|(1.5)]@/^-1.3cm/"7"^{x_m^{-1}}
"7":@[|(1.5)]"8"_{w^{-1}}
!{(1.25,0)}*+{\ldots}="9"
!{(3.4,0)}*+{\ldots}="10"
!{(-.1,-.1)}*+{\alpha}
!{(5.4,-.1)}*+{\omega}
!{(3.35,-.15)}*+{B_{1}}
%!{(3.5,-.2)}*+{B_{N-1}}
!{(4.45,-.15)}*+{B_{N}}
}}
\caption{\label{fig:KOP1}The graph $\Gamma(w_1,\ldots,w_k,w,N)$.}
\end{figure}
It is clear that the problem has a positive answer if and only if $1\in L(\Gamma)$. By Proposition~\ref{pr:goodloop_existence}, it suffices to check whether the graph $\Delta=\mathcal F(\CC^{O(\log (|w|+\sum|w_i|+mN))}(\Gamma))$ contains the edge $\alpha \stackrel{\varepsilon}{\rightarrow} \omega$.
\end{proof}

Now we turn to solving $\KOP2$. Again, by Thereom~\ref{th:PTime_Bulitko}, it is enough to solve $\SSOP2$, which is achieved using the following statement.
\begin{theorem}
\label{th:hyp_kop2} Let $G$ be a hyperbolic group given by a finite presentation $\langle X\mid R\rangle$.
There exists a polynomial time algorithm that, given $g_1,\ldots,g_k,g\in G$ and a unary $N\in\mathbb N\cup\{0\}$,
finds $\varepsilon_1, \ldots, \varepsilon_k \in\{0,1\}$ such that $g_1^{\varepsilon_1} \ldots g_k^{\varepsilon_k}$ belongs to the segment $[1,g]$, and the distance between $g$ and $g_1^{\varepsilon_1} \ldots g_k^{\varepsilon_k}$ in the Cayley graph $Cay(G,X)$ does not exceed $N$, or outputs {\em No solutions} if no such $\varepsilon_1,\ldots,\varepsilon_k$ exist.
\end{theorem}
\begin{proof}
%Consider the graph $\Delta$ obtained in the proof of Theorem~\ref{th:hyp_kop1} above.
As in proof of Theorem~\ref{th:hyp_kop1}, we only need solve the corresponding decision problem.

Recall that hyperbolic groups are strongly geodesically automatic (\cite{Cannon:1984}), which means that they possess an automatic structure, where the language $\mathcal L$ accepted by the word acceptor is the set of all geodesic words. Recall further that an equality checker (see, for example,~\cite{Gersten-Short:1990}) for an automatic group $G$ is the automaton that accepts the subset $\{(u,v)\mid u=_G v\}$ of $\mathcal L\times\mathcal L$.

For a given $g\in G$, one can construct in polynomial time an automaton $M_g$ that accepts all geodesic words equal to $g$ in $G$. Indeed, this can be done by building an automaton product of the equality checker (see~\cite{Gersten-Short:1990}) and the automaton $\Gamma_g$ in Fig.~\ref{fig:geodesic}, where $w=y_1y_2\ldots y_{|g|}$ is a geodesic word representing $g$ in generators $X=\{x_1,\ldots,x_m\}$. Further, in the automaton $M_g$ we mark every vertex that is distance at most $N$ from the terminal one.

\begin{figure}[htb]
 \centering
 \includegraphics[height=0.8in]{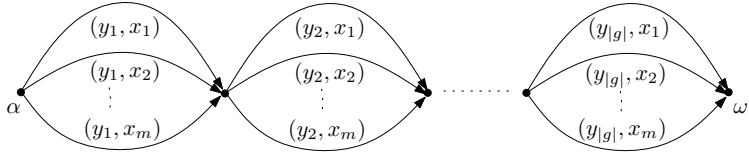}
 \caption{The graph $\Gamma_g$.}\label{fig:geodesic}
\end{figure}

\begin{figure}[h]
\centerline{
\xygraph{
!{<0cm,0cm>;<2cm,0cm>:<0cm,2cm>::}
!{(0,0)}*+{\bullet}="0"
!{(1,0)}*+{\bullet}="1"
!{(2,0)}*+{\bullet}="2"
!{(3,0)}*+{\bullet}="3"
!{(4,0)}*+{\bullet}="4"
%!{(5,0)}*+{\bullet}="5"
"0":@[|(1.5)]@/^0.5cm/"1"^{w_k^{-1}}
"0":@[|(1.5)]"1"_{\varepsilon}
"1":@[|(1.5)]@/^0.5cm/"2"^{w_{k-1}^{-1}}
"1":@[|(1.5)]"2"_{\varepsilon}
"3":@[|(1.5)]@/^0.5cm/"4"^{w_1^{-1}}
"3":@[|(1.5)]"4"_{\varepsilon}
%"4":@[|(1.5)]"5"_{w^{-1}}
!{(2.5,0)}*+{\ldots}="6"
!{(-.1,-.1)}*+{\alpha}
!{(4.1,-.1)}*+{\omega}
}}
\caption{\label{fig:ssp_no_w}The graph $\Gamma$.}
\end{figure}

Let $\Gamma$ be the automaton displayed in Fig.~\ref{fig:ssp_no_w}. We obtain an automaton $\Delta=\Delta(g_1,g_2,\ldots,g_k,g,N)$ by attaching copies $\Gamma_1,\Gamma_2,\ldots$ of $\Gamma$ to every marked vertex of $M_g$, as in Fig.~\ref{fig:pokerface_fish}. We assign the initial vertex of $\Delta$ to be the initial vertex $\alpha$ of $M_g$, and the set of terminal vertices to consist of the terminal vertices $\omega_1,\omega_2,\ldots$ of the copies $\Gamma_1,\Gamma_2,\ldots$.
\begin{figure}[htb]
 \centering
 \includegraphics[height=1.1in]{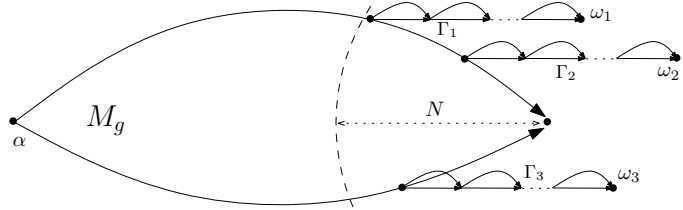}
 \caption{The graph $\Delta$.}\label{fig:pokerface_fish}
\end{figure}

Notice that the problem has a positive answer if and only if $1\in L(\Delta)$. By Proposition~\ref{pr:goodloop_existence}, it is enough to check whether there is the edge $\alpha \stackrel{\varepsilon}{\rightarrow} \omega_j$ in $\mathcal F(\CC^{O(\log l)}(\Delta))$, where $l$ is the number of edges in $\Delta$. It is only left to note that the number of edges in $M_g$ and $\Gamma$ is polynomial in $\sum|g_i|+|g|+N$, therefore so is $l$.
\end{proof}

\begin{corollary}
Let $G$ be a hyperbolic group. Then $\KOP$, $\SSOP1$, $\SSOP2$, $\KOP1$, $\KOP2$ for $G$ are in $\P$.
\end{corollary}
%\section{Open problems}
%
%
%\begin{problem}
%Complexity of $\SSP$ and  $\KP$  in right angled Artin (graph) groups.
%\end{problem}
%
%\begin{problem}
%Complexity of $\SSP$ and  $\KP$  in $F \times F$, where $F$ is a free non-abelian group of finite rank.
%\end{problem}
%
%\begin{problem}
%Complexity of $\SSP$ and $\KP$ in Grigorchuk group $\Gamma$ and Thompson group $F$.
%\end{problem}
%
%\begin{problem}
%Complexity of $\SSP$ and $\KP$ in  Thompson group $F$.
%\end{problem}
%
%
%\begin{problem}
%Complexity of $\BSMP$ in the following groups:
%\begin{itemize}
%\item right angled Artin (graph) groups.
%\item $F \times F$, where $F$ is a free non-abelian group of finite rank.
%\item Grigorchuk group $\Gamma$
%\item Thompson group $F$.
%\end{itemize}
%\end{problem}
%
%\begin{problem}
%Complexity of $\KP$ and $\BSMP$ in  finitely generated nilpotent groups.
%\end{problem}
%
%\begin{problem}
%Complexity of $\SSP$, $\KP$ and $\BSMP$ in  free products with amalgamation and HNN extensions (under proper conditions on the factors and the edge group).
%\end{problem}

\section{Bounded submonoid membership problem for $F_2\times F_2$}\label{se:Mikh}

We proved in Section~\ref{sec:bounded_bsmp} that the bounded submonoid problem is
decidable in any hyperbolic group $G$ in polynomial time. In this
section we show that taking a direct product does not preserve
$\P$-time decidability of $\BSMP$ unless $\P=\NP$.
In fact, we prove a stronger result.
We show that there exists a (fixed!) subgroup $H = \gp{h_1,\ldots,h_k}$
in $F_2\times F_2$ with $\NP$-complete bounded membership problem.

\medskip
\noindent
{\bf The bounded $\GWP$ for a fixed subgroup $H = \gp{h_1,\ldots,h_k} \le G$:}
Given $g\in G$ and unary $1^n\in\MN$ decide if $g$ can be expressed
as a product of the form $g=h_{i_1}^{\pm 1}h_{i_2}^{\pm 1}\cdots h_{i_l}^{\pm 1}$,
where $l\le n$ and $1\le i_1,\ldots, i_l\le k$.
% of at most $n$ generators $h_i^{\pm 1}$.
\medskip

Similar to Proposition~\ref{le:FiniteX_independence},
one can show that complexity of the bounded membership problem does not
depend on a finite generating set for $G$, and, hence,
we can denote this problem $\BGWP(G;h_1,\ldots,h_k)$.

\begin{proposition}
$\BGWP(G;h_1,\ldots,h_k)$ is $\P$-time reducible to $\BSMP(G)$.
Therefore, if $\BGWP(G;h_1,\ldots,h_k)$ is $\NP$-complete and
the word problem in $G$ is in $\P$, then $\BSMP(G)$ is $\NP$-complete.
\end{proposition}

\begin{proof}
$(g,1^n)$ is a positive instance of $\BGWP(G;h_1,\ldots,h_k)$ if and only if
$(h_1,\ldots,h_k,\linebreak h_1^{-1},\ldots,h_k^{-1},g,1^n)$ is a positive instance of $\BSMP(G)$.
\end{proof}

Below we prove that there exists a subgroup $H = \gp{h_1,\ldots,h_k}$
in $F_2\times F_2$ with $\NP$-complete $\BGWP(F_2\times F_2;h_1,\ldots,h_k)$.
In our argument we employ the idea used by Olshanskii and Sapir
in~\cite[Theorems 2 and 7]{Olshanskii_Sapir:2001}
to investigate subgroup distortions in $F_2\times F_2$. The argument follows
Mikhailova's construction of a subgroup of $F_2\times F_2$
with undecidable membership problem.
%The group $F_2\times F_2$ was the first example
%of a group with decidable word problem and undecidable membership problem.
We briefly outline that construction as described in~\cite{Mihailova}.

Let $G=\langle X\mid R\rangle$ be a finitely presented group.
%with undecidable word problem (e.g., \cite{Novikov:1955,Boone:1958}).
We may assume that both sets $X$ and $R$ are symmetric, i.e.,
$X=X^{-1}$ and $R=R^{-1}$.
%\begin{itemize}
%    \item
%$X=X^{-1}$,
%    \item
%$R=R^{-1}$.
%\end{itemize}
Define a set:
\begin{equation}\label{eq:Mikhailova_gens}
D_G=\{(r,1)\mid r\in R\}\cup\{(x,x^{-1})\mid x\in X\} \ \ \subset \ \ F(X)\times F(X).
\end{equation}
Let $H$ be a subgroup of $F(X)\times F(X)$ generated by $D_G$.
Then for any $w\in F(X)$:
\begin{equation}\label{eq:Mikhailova_constr}
(w,1) \in H \ \ \Leftrightarrow\ \ w=1 \mbox{ in }G.
\end{equation}
In more detail, the following lemma is true.

\begin{lemma}[\cite{Mihailova}]\label{lemma:Mikh}
Let $w=w(X)$. If
    $$(w,1)=(u_1,v_1)(u_2,v_2)\ldots (u_n,v_n) \mbox{ for some } (u_i,v_i)\in D_G,$$
then the word $u_1\ldots u_n$ is of the form
    $$w_0 r_1 w_1 r_2 w_2 \ldots w_{m-1}r_m w_m \mbox{ for some } w_i\in F(X),\ r_i\in R$$
satisfying $w_0w_1\ldots w_m =_{F(X)} 1$ and, hence,
    $$w =_{F(X)} \prod_{i=1}^m (w_0\ldots w_{i-1}) r_i (w_0\ldots w_{i-1})^{-1}.$$
\end{lemma}

Moreover, by~\cite[Lemma 1]{Olshanskii_Sapir:2001}
we may assume that $|w_0|+\ldots+|w_m|\le 4E$, where $E$ is the number of edges in
the minimal van Kampen diagram for $w$ over $\langle X\mid R\rangle$. Denoting $C=\max\{|r|: r\in R\}$, we get
$|w_0|+\ldots+|w_m|\le 4(Cm+|w|)$,
so
\begin{equation}\label{eq:quadratic}
n\le m+2(|w_0|+|w_1|+\ldots+|w_{m}|)\le m+8(Cm+|w|),
\end{equation}
i.e. the number of elements of $D_G$ necessary to represent $(w,1)$ is bounded by a polynomial (in fact, linear) function of $|w|$ and $m$.

\begin{lemma}\label{le:wp_reduces_to_bsmp}
Let $\gpr{X}{R}$ be finite presentation of a group $G$
and $D_G \subset F(X)\times F(X)$ the set given by~(\ref{eq:Mikhailova_gens}).
If the isoperimetric function for $\gpr{X}{R}$ is bounded by a polynomial $p$, then
the word problem in $G$ is $\P$-time reducible to $\BGWP(F(X)\times F(X);D_G)$.
\end{lemma}

\begin{proof}
As above, let $C=\max\{|r|: r\in R\}$.
For an arbitrary $w\in F(X)$ compute
    $$n=p(|w|)+8(Cp(|w|)+|w|).$$
Now it easily follows from Lemma \ref{lemma:Mikh} and inequality (\ref{eq:quadratic})
that $w=1$ in $G$ if and only if $((w,1), 1^n)$ is a positive instance of $\BGWP(F(X)\times F(X);D_G)$.
\end{proof}

\begin{theorem}\label{th:MikhNP}
There is a finitely generated subgroup $H=\langle h_1,\ldots,h_k\rangle $
in $F_2\times F_2$ such that $\BGWP(F_2\times F_2; h_1,\ldots,h_k)$ is $\NP$-complete.
\end{theorem}

\begin{proof}
It is showed %in~\cite[Section 4.4.A]{Sapir:2012}
in~\cite{Sapir-Birget-Rips:2002}
that there exists a finitely presented group $G$ with
$\NP$-complete word problem and polynomial Dehn function.
Let $D_G = \{h_1,\ldots,h_k\}$ be a subset of $F(X) \times F(X)$ defined by~(\ref{eq:Mikhailova_gens}).
By Lemma~\ref{le:wp_reduces_to_bsmp}, $\BGWP(F(X)\times F(X); D_G)$ is $\NP$-hard.
Since $F_2\times F_2$ contains subgroup isomorphic to $F(X) \times F(X)$,
$\BGWP(F_2\times F_2; D_G)$ is also $\NP$-hard.
It is only left to note that the word problem in $F_2\times F_2$ is $\P$-time decidable,
so $\BGWP(F_2\times F_2; D_G)$ is $\NP$-complete.
\end{proof}

\begin{corollary}\label{co:F2_BMP}
If $G$ contains $F_2\times F_2$ as a subgroup,
then there exists $\{h_1,\ldots,h_k\} \subseteq G$
such that $\BGWP(G;h_1,\ldots,h_k)$ and $\BSMP(G)$ are $\NP$-hard.
If, in addition, the word problem in $G$ is $\P$-time decidable,
then $\BGWP(G;h_1,\ldots,h_k)$ and $\BSMP(G)$ are $\NP$-complete.
\end{corollary}

\begin{corollary} Linear groups $GL(\ge\!4, \mathbb Z)$, braid groups and graph groups whose graph contains
an induced square $C_4$ have $\NP$-complete $\BGWP$ and $\BSMP$.
\end{corollary}

%\bibliography{../../main_bibliography}
\bibliography{main_bibliography}

\begin{thebibliography}{10}

\bibitem{Alonso-etal:1991}
J.~{Alonso}, T.~{Brady}, D.~{Cooper}, V.~{Ferlini}, M.~{Lustig}, M.~{Mihalik},
  M.~{Shapiro}, and H.~{Short}.
\newblock {Notes on word hyperbolic groups}.
\newblock In {\em Group Theory from a geometrical viewpoint}, pages 3--63.
  World Scientific, 1991.

\bibitem{Baumslag:1962}
G.~{Baumslag}.
\newblock {On generalized free products}.
\newblock {\em Math. Z.}, 78:423--438, 1962.

\bibitem{Baumslag:1972}
G.~{Baumslag}.
\newblock {A finitely presented metabelian group with a free abelian derived
  group of infinite rank}.
\newblock {\em Proc. Amer. Math. Soc.}, 35:61--62, 1972.

\bibitem{Baumslag:1973}
G.~{Baumslag}.
\newblock {Subgroups of finitely presented metabelian groups}.
\newblock {\em J. Australian Math. Soc.}, 16:98--110, 1973.

\bibitem{Baumslag_Miasnikov_Remeslennikov:1999}
G.~{Baumslag}, A.~G. {Miasnikov}, and V.~{Remeslennikov}.
\newblock {Algebraic geometry over groups I. Algebraic sets and ideal theory}.
\newblock {\em J. Algebra}, 219:16--79, 1999.

\bibitem{Baumslag_Miasnikov_Remeslennikov:2002}
G.~{Baumslag}, A.~G. {Miasnikov}, and V.~{Remeslennikov}.
\newblock {Discriminating completions of hyperbolic groups}.
\newblock {\em Geometriae Dedicata}, 92:115--143, 2002.

\bibitem{Bogopolskii-Gerasimov:1995}
O.~{Bogopolski} and V.~{Gerasimov}.
\newblock {Finite subgroups of hyperbolic groups}.
\newblock {\em Algebra i Logika}, 34:343--345, 1995.

\bibitem{Brady:2000}
N.~{Brady}.
\newblock Finite subgroups of hyperbolic groups.
\newblock {\em Int. J. Algebra Comput.}, 10:399--405, 2000.

\bibitem{Bryant-Roman'kov:1999}
R.~M. {Bryant} and V.~A. {Roman'kov}.
\newblock {Automorphism groups of relatively free groups}.
\newblock {\em Math. Proc. Cambridge Philos. Soc.}, 127:411--424, 1999.

\bibitem{Bulitko:1970}
V.~K. {Bulitko}.
\newblock {Equations and inequalities in a free group and a free semigroup}.
\newblock {\em Tul. Gos. Ped. Inst. Uchenye. Zap. Mat. Kaf.}, 2:242--252, 1970.

\bibitem{CFP}
J.~W. {Cannon}, W.~J. {Floyd}, and W.~R. {Parry}.
\newblock {Introductory notes on Richard Thompson's groups}.
\newblock {\em L'Enseignement Mathematique}, 42(2):215--256, 1996.

\bibitem{Cannon:1984}
James~W. {Cannon}.
\newblock The combinatorial structure of cocompact discrete hyperbolic groups.
\newblock {\em Geometriae Dedicata}, 16:123--148, 1984.

\bibitem{Cleary:2005}
S.~{Cleary}.
\newblock Distortion of wreath products in some finitely presented groups.
\newblock {\em Pac. J. Math.}, 228(1):53--61, 2006.

\bibitem{Dasgupta-Papadimitriou-Vazirani:2006}
S.~{Dasgupta}, C.~{Papadimitriou}, and U.~{Vazirani}.
\newblock {\em Algorithms}.
\newblock McGraw-Hill Science, 2006.

\bibitem{GJ}
M.~{Garey} and J.~{Johnson}.
\newblock {\em Computers and Intractability: A Guide to the Theory of
  NP-Completeness}.
\newblock W. H. Freeman, 1979.

\bibitem{Gersten-Short:1990}
S.~{Gersten} and H.B. {Short}.
\newblock {Small cancellation theory and automatic groups}.
\newblock {\em Inventiones mathematicae}, 102, 1990.

\bibitem{Gromov_pgrowth:1981}
M.~{Gromov}.
\newblock {Groups of polynomial growth and expanding maps}.
\newblock {\em Publ. Math. IHES}, 53:53--73, 1981.

\bibitem{Gromov_hyperbolic}
M.~{Gromov}.
\newblock {Hyperbolic groups}.
\newblock In {\em Essays in group theory}, volume~8 of {\em MSRI Publications},
  pages 75--263. Springer, 1985.

\bibitem{Holt-Rees:2001}
D.~{Holt} and S.~{Rees}.
\newblock {Solving the word problem in real time}.
\newblock {\em J. London Math. Soc.}, 63(2):623--639, 2001.

\bibitem{Kellerer-Pferschy-Pisinger:2004}
H.~{Kellerer}, U.~{Pferschy}, and D.~{Pisinger}.
\newblock {\em Knapsack Problems}.
\newblock Springer, 2004.

\bibitem{Kharlampovich-Lysenok-Myasnikov-Touikan:2010}
O.~{Kharlampovich}, I.~Lysenok, A.~G. Myasnikov, and N.~Touikan.
\newblock {The solvability problem for quadratic equations over free groups is
  NP-complete}.
\newblock {\em Theor. Comput. Syst.}, 47:250--258, 2010.

\bibitem{Kharlampovich_Myasnikov:1998(1)}
O.~{Kharlampovich} and A.~{Myasnikov}.
\newblock {Irreducible affine varieties over a free group. I: Irreducibility of
  quadratic equations and Nullstellensatz}.
\newblock {\em J. Algebra}, 200(2):472--516, 1998.

\bibitem{KharlampovichMyasnikovLyutikova:1999}
O.~{Kharlampovich}, A.~{Myasnikov}, and {Lyutikova} E.
\newblock {Equations over $Q$-completions of hyperbolic groups}.
\newblock {\em Trans. of AMS}, 351(7):2961--2978, 1999.

\bibitem{LohreySteinberg:2011}
M.~{Lohrey} and B.~{Steinberg}.
\newblock Tilings and submonoids of metabelian groups.
\newblock {\em Theory Comp. Syst.}, 48:411--427, 2011.

\bibitem{Merkle-Hellman:1978}
R.~{Merkle} and M.~{Hellman}.
\newblock {Hiding information and signatures in trapdoor knapsacks}.
\newblock {\em Inform. Theory, IEEE Trans.}, 24:525--530, 1978.

\bibitem{Miasnikov_Nikolaev:2011}
A.~G. {Miasnikov} and A.~{Nikolaev}.
\newblock {Verbal subgroups of hyperbolic groups have infinite width}.
\newblock preprint. Available at \url{http://arxiv.org/abs/1107.3719}, 2011.

\bibitem{Miasnikov_Remeslennikov:1996}
A.~G. {Miasnikov} and V.~{Remeslennikov}.
\newblock {Exponential groups II: extensions of centralizers and tensor
  completion of CSA-groups}.
\newblock {\em Int. J. Algebra Comput.}, 6(6):687--711, 1996.

\bibitem{Miasnikov_Remeslennikov_Serbin:2005}
A.~G. {Miasnikov}, V.~{Remeslennikov}, and D.~{Serbin}.
\newblock Regular free length functions on lyndon's free $\mathbb{Z}[t]$-group
  $f^{\mathbb{z}[t]}$.
\newblock In {\em Algorithms, Languages, Logic}, volume 378 of {\em
  Contemporary Mathematics}, pages 37--77. American Mathematical Society, 2005.

\bibitem{Miasnikov_Remeslennikov_Serbin:2006}
A.~G. {Miasnikov}, V.~{Remeslennikov}, and D.~{Serbin}.
\newblock {Fully residually free groups and graphs labeled by infinite words}.
\newblock {\em Int. J. Algebra Comput.}, 66(4):689--737, 2006.

\bibitem{Miasnikov_Romankov_Ushakov_Vershik:2010}
A.~G. {Miasnikov}, V.~{Romankov}, A.~{Ushakov}, and A.~{Vershik}.
\newblock The word and geodesic problems in free solvable groups.
\newblock {\em Trans. Amer. Math. Soc.}, 362:4655--4682, 2010.

\bibitem{MSU_book:2011}
A.~G. {Miasnikov}, V.~{Shpilrain}, and A.~{Ushakov}.
\newblock {\em Non-Commutative Cryptography and Complexity of Group-Theoretic
  Problems}.
\newblock Mathematical Surveys and Monographs. AMS, 2011.

\bibitem{MU1}
A.~G. {Miasnikov} and A.~{Ushakov}.
\newblock {Random van Kampen diagrams and algorithmic problems in groups}.
\newblock {\em Groups Complex. Cryptol.}, 3:121--185, 2011.

\bibitem{Mihailova}
K.~A. {Mihailova}.
\newblock {The occurrence problem for direct products of groups}.
\newblock {\em Dokl. Akad. Nauk SSSR}, 119:1103--1105, 1958.

\bibitem{Nikolaev:thesis}
A.~{Nikolaev}.
\newblock {\em Membership problem in groups acting freely on non-Archimedean
  trees}.
\newblock PhD thesis, McGill University, 2010.

\bibitem{Nikolaev_Serbin:2011(2)}
A.~{Nikolaev} and D.~{Serbin}.
\newblock {Membership problem in groups acting freely on $\mathbb{Z}^n$-trees}.
\newblock Submitted. Available at \url{http://arxiv.org/abs/1107.0943}, 2011.

\bibitem{Odlyzko:1990}
A.~Odlyzko.
\newblock The rise and fall of knapsack cryptosystems.
\newblock In {\em In Cryptology and Computational Number Theory}, pages 75--88.
  AMS, 1990.

\bibitem{Olshanskii_Sapir:2001}
A.~{Ol'shanskii} and M.~{Sapir}.
\newblock {Length and area functions on groups and quasi-isometric Higman
  embeddings}.
\newblock {\em Internat. J. Algebra Comput.}, 11(2):137--170, 2001.

\bibitem{Olshanskii:1989}
A.~Yu. {Ol'shanskii}.
\newblock Diagrams of homomorphisms of surface groups.
\newblock {\em Sibirsk. Mat. Zh.}, 30:150--171, 1989.

\bibitem{Olshanskii:1991}
A.~Yu. {Ol'shanskii}.
\newblock Periodic quotients of hyperbolic groups.
\newblock {\em Mat. Zametki}, 182:543--567, 1991.

\bibitem{Papa}
C.~{Papadimitriou}.
\newblock {\em Computation Complexity}.
\newblock Addison-Wesley, 1994.

\bibitem{Papadimitriou-Steiglitz:1998}
C.~{Papadimitriou} and K.~{Steiglitz}.
\newblock {\em Combinatorial optimization: Algorithms and Complexity}.
\newblock Dover Publications, 1998.

\bibitem{Remeslennikov:1973}
V.~{Remeslennikov}.
\newblock {On finitely presented groups}.
\newblock In {\em Fourth All- Union Symposium on the Theory of Groups 1973},
  pages 164--169, Novosibirsk, USSR, 1973.

\bibitem{Rips:1982}
E.~{Rips}.
\newblock {Subgroups of small cancellation groups}.
\newblock {\em Bull. London Math. Soc.}, 14:45--47, 1982.

\bibitem{Romanovskii}
N.~S. {Romanovskii}.
\newblock {The occurrence problem for extensions of abelian by nilpotent
  groups}.
\newblock {\em Sib. Math. J.}, 21:170--174, 1980.

\bibitem{Sapir-Birget-Rips:2002}
M.V. {Sapir}, J.-C. {Birget}, and E.~{Rips}.
\newblock Isoperimetric and isodiametric functions of groups.
\newblock {\em Ann. Math.}, 156(2):345--466, 2002.

\bibitem{Schupp:2003}
P.~{Schupp}.
\newblock {Coxeter groups, 2-completion, perimeter reduction and subgroup
  separability}.
\newblock {\em Geometriae Dedicata}, 96(1):179--198, 2003.

\bibitem{Shamir:1984}
A.~{Shamir}.
\newblock {A polynomial-time algorithm for breaking the basic Merkle - Hellman
  cryptosystem}.
\newblock {\em Inform. Theory, IEEE Trans}, 30(5):699--704, 1984.

\bibitem{SU1}
V.~{Shpilrain} and A.~{Ushakov}.
\newblock Thompson's group and public key cryptography.
\newblock In {\em Applied Cryptography and Network Security -- ACNS 2005},
  volume 3531 of {\em Lecture Notes Comp. Sc.}, pages 151--164. Springer, 2005.

\bibitem{Umirbaev:1997}
U.~{Umirbaev}.
\newblock {Occurrence problem for free solvable groups}.
\newblock {\em Algebra and Logic}, 34:112--124, 1995.

\bibitem{Ushakov:thesis}
A.~{Ushakov}.
\newblock {\em Fundamental Search Problems in Groups}.
\newblock PhD thesis, CUNY/Graduate Center, 2005.

\bibitem{Venkatesan-Rajagopalan:1992}
R.~Venkatesan and S.~Rajagopalan.
\newblock Average case intractability of matrix and diophantine problems
  (extended abstract).
\newblock In S.~Rao Kosaraju, Mike Fellows, Avi Wigderson, and John~A. Ellis,
  editors, {\em STOC}, pages 632--642. ACM, 1992.

\bibitem{Wolf}
J.~{Wolf}.
\newblock {Growth of finitely generated solvable groups and curvature of
  Riemannian manifolds}.
\newblock {\em Journal of Differential Geometry}, 2:421--446, 1968.

\end{thebibliography}

\end{document}